\documentclass[a4paper,12pt,reqno]{amsart}
\usepackage{amsmath,amsthm,amssymb,latexsym,epsfig,graphicx,subfigure}
\newtheorem{theorem}{Theorem}[section]
\newtheorem{lemma}[theorem]{Lemma}

\newtheorem{proposition}[theorem]{Proposition}

\theoremstyle{definition}
\newtheorem{definition}[theorem]{Definition}
\newtheorem{example}[theorem]{Example}

\newtheorem{remark}[theorem]{Remark}

\DeclareMathOperator{\dimH}{dim_H}
\DeclareMathOperator{\dimb}{dim_B}
\DeclareMathOperator{\dimp}{dim_p}
\DeclareMathOperator{\udimb}{\overline{dim}_B}
\DeclareMathOperator{\ldimb}{\underline{dim}_B}
\DeclareMathOperator{\Id}{Id}

\numberwithin{equation}{section}

\newcommand{\ba}{{\bf a}}
\newcommand{\bi}{{\bf i}}
\newcommand{\bj}{{\bf j}}
\newcommand{\bk}{{\bf k}}

\begin{document}

\title[random code tree fractals]{Dimensions of random affine
                                  code tree fractals}

\author[EJ]{Esa J\"arvenp\"a\"a}
\author[MJ]{Maarit J\"arvenp\"a\"a}
\author[AK]{Antti K\"aenm\"aki}
\author[HK]{Henna Koivusalo}
\author[\"OS]{\"Orjan Stenflo}
\author[VS]{Ville Suomala}

\address{Department of Mathematical Sciences, P.O. Box 3000, 
  90014 University of Oulu, Finland}
\email{esa.jarvenpaa@oulu.fi}

\address{Department of Mathematical Sciences, P.O. Box 3000, 
  90014 University of Oulu, Finland}
\email{maarit.jarvenpaa@oulu.fi}

\address{Department of Mathematics and Statistics, P.O. Box 35, 
  40014 University of Jyv\"askyl\"a, Finland} 
\email{antti.kaenmaki@jyu.fi}

\address{Department of Mathematical Sciences, P.O. Box 3000, 
  90014 University of Oulu, Finland}
\email{henna.koivusalo@oulu.fi}

\address{Department of Mathematics, Uppsala University, P.O. Box 480, 
  75106 Uppsala, Sweden} 
\email{stenflo@math.uu.se} 
 
\address{Department of Mathematical Sciences, P.O. Box 3000, 
  90014 University of Oulu, Finland} 
\email{ville.suomala@oulu.fi}

\begin{abstract}
We calculate the almost sure Hausdorff dimension for a general class of random 
affine planar code tree fractals. The set of probability measures
describing the randomness includes natural measures in random $V$-variable and 
homogeneous Markov constructions.
\end{abstract}

\thanks{We acknowledge the support of the Centre of Excellence in Analysis and 
Dynamics Research funded by the Academy of Finland. HK is also grateful for 
Jenny and Antti Wihuri foundation. \"OS thanks the Esseen foundation. VS 
acknowledges the support of the Academy of Finland project \# 126976.} 

\maketitle

\section{Introduction}\label{intro}

The systematic study of dimensions of the attractors of iterated function 
systems was initiated by Hutchinson \cite{H}. He proved the formula for  
Hausdorff dimensions of attractors of self-similar iterated function systems 
satisfying the 
open set condition. Since then, iterated function systems have been studied 
extensively and nowadays there exists a huge literature on them. Here we 
mention only a few results relevant to our purposes. In the case of affine 
maps, there is no counterpart to the open set condition guaranteeing that the 
dimension of the attractor is given by a concrete formula. However, in 
\cite{F88} Falconer proved that writing the affine maps as 
$F_i(x)=T_i(x)+a_i$, the dimension of the attractor is Lebesgue almost surely 
independent of the translation vectors $a_i$ provided that the norms of the 
linear parts $T_i$ are less than $\tfrac 13$. Furthermore, the dimension is 
given by the unique zero of a natural pressure.
Solomyak \cite{S98} verified that $\tfrac 13$ may be 
replaced by $\tfrac 12$, and an example of Edgar \cite{E} (see also 
\cite{Er,PU,SS}) shows that $\tfrac 12$ is the best possible upper bound in 
this setting. 

In this paper we verify that the Hausdorff dimension is almost surely 
independent
of the translation vectors for general affine code tree fractals, that is,
for iterative constructions where the families of contractions may vary (see 
section~\ref{preli} for the definition). We also illustrate by
examples that for general affine code tree fractals the dimension is not 
always given by the zero of the pressure.   

A natural way to generalize deterministic iterated function systems
is to add randomness to the construction. This can be done in various different
ways.
Jordan, Pollicott and Simon \cite{JPS} studied a fixed affine iterated function
system having a small independent random perturbation in translation vectors
at each step of the construction. In this case
no nontrivial upper bound for the contraction ratios is needed in the analogue
of Falconer's result \cite{F88}. In \cite{FM09} Falconer and Miao studied
random subsets of a fixed self-affine fractal. The randomness is introduced by
choosing at each step of the construction a random subfamily of the 
original function system independently. 
They proved that the dimension of the random 
subset is almost surely given by the unique zero of the expected pressure.

Both in \cite{JPS} and \cite{FM09} the randomness is quite strong in the sense
that there is total independence both in space, that is, between different 
nodes at a fixed construction level, and in scale or time, that 
is, once a node is chosen, its descendants are selected independently of the 
previous history. We will consider probability distributions which  
have certain independence only in
time direction, more precisely, there exists almost surely
an infinite sequence of neck levels $N_n$ where all the sub-code-trees starting
from level $N_n$ are the same and the events depending only on the construction
before a neck level $N_k$ are independent of those depending only on the
construction after $N_k$. The structure between neck levels can be chosen quite 
freely. In other words, our construction is locally random but globally almost
homogeneous. 

A special case of our setting (which in fact is the main motivation for our 
study) is the natural probability measure $P$ of choosing random $V$-variable 
fractals considered by Barnsley, Hutchinson and Stenflo
in \cite{BHS2005,BHS2008,BHS08}. A $V$-variable fractal is a code tree fractal
where at each level of the code tree there are at most $V$ different 
sub code trees. In \cite{BHS08} a formula for the $P$-almost sure Hausdorff 
dimension of random self-similar $V$-variable fractals satisfying a uniform 
open set condition is proved. We extend this result in two ways: by replacing
similarities with affine maps and by considering a general class of probability
measures including the measure $P$ introduced in \cite{BHS08}.

The paper is organized as follows: In section \ref{preli} we give basic 
definitions and recall some well-known results. In section \ref{det} we 
extend Falconer's result concerning the almost sure constancy of the Hausdorff 
dimension \cite{F88} to deterministic code tree fractals (Theorem \ref{0510}) 
and give examples illustrating reasons why the Hausdorff 
dimension is not given by the zero of the natural pressure function.  
Sections \ref{random} and \ref{dimension} are dedicated to random code tree 
fractals and contain our main results: 
In Theorem \ref{pexists} we prove that, under quite weak assumptions on the
probability measure on the set of code trees having a neck structure,  
the zero of the pressure exists and is independent of
the code tree almost surely. 
In Theorem \ref{main} we verify that, under certain additional assumptions, the
zero of the pressure is equal to the typical Hausdorff dimension
of the code tree fractal almost surely. Finally, in 
Proposition~\ref{selfsimilar} it is observed that for self-similar code tree 
fractals the additional assumptions are not needed.

\section{Preliminaries}\label{preli}

We begin by defining a code tree fractal. Our definition is similar to the one
in \cite{BHS2008}. We denote by $\Lambda$ an index set. Let $D\in\mathbb N$ and
let ${\bf F}=\{F^\lambda\mid\lambda\in\Lambda\}$ be a family of iterated 
function systems such that 
$F^\lambda=\{f_1^\lambda,\dots,f_{M_\lambda}^\lambda\}$. Here for all 
$i=1,\dots,M_\lambda$ the map $f_i^\lambda\colon\mathbb R^D\to\mathbb R^D$ is
defined by 
$f_i^\lambda(x)= T_i^\lambda(x)+a_i^\lambda$, where $a_i^\lambda\in\mathbb R^D$
and $T_i^\lambda$ is a non-singular linear mapping with 
$\sup_{\lambda\in\Lambda,i=1,\dots,M_\lambda}\Vert T_i^\lambda\Vert<1$ and
$M=\sup_{\lambda\in\Lambda}M_\lambda<\infty$. Setting $I=\{1,\dots,M\}$, 
the length
of a word $\tau\in I^k$ is $|\tau|=k$. Consider
a function $\omega\colon\bigcup_{k=0}^\infty I^k\to\Lambda$, where
$I^0=\emptyset$. Let $\Sigma^{\omega}_{*}\subset\bigcup_{k=0}^\infty I^k$ be the 
unique set satisfying the following conditions:
\begin{itemize}
\item $\emptyset\in\Sigma^\omega_*$,
\item if $i_1\cdots i_k\in\Sigma^\omega_*$ and 
   $\omega(i_1\cdots i_k)=\lambda$, then $i_1\cdots i_kl\in\Sigma^\omega_*$ 
   for all $l\leq M_\lambda$ and $i_1\cdots i_kl\notin\Sigma^\omega_*$
   for any $l>M_\lambda$,
\item if $i_1\cdots i_k\notin\Sigma^\omega_*$, then 
   $i_1\cdots i_kl\notin\Sigma^\omega_*$ for any $l$.
\end{itemize}
The function $\omega$ restricted to $\Sigma^\omega_*$ is called 
an ${\bf F}$-valued code tree 
and the set of all ${\bf F}$-valued code trees is denoted by $\Omega$.
Equip $I^{\mathbb N}$ with the product topology. Let 
$\Sigma^\omega=\{\bi=i_1i_2\cdots\in I^{\mathbb N}\mid 
  i_1\cdots i_n\in\Sigma^\omega_*\text{ for all }n\in\mathbb N\}$ be the 
compact set of infinite paths corresponding to a code tree $\omega\in\Omega$.
For any $k\in\mathbb N$ and $\bi\in\Sigma^\omega\cup\bigcup_{j=k}^\infty I^j$, 
let $\bi_k=i_1\cdots i_k$ be the initial word of $\bi$ with length $k$. 
We use the notations
\[
f^\omega_{\bi_k}=f_{i_1}^{\omega(\emptyset)}\circ f_{i_2}^{\omega(i_1)} 
  \circ\dotsb\circ f_{i_k}^{\omega(i_1\cdots i_{k-1})}\text{ and }
T_{\bi_k}^\omega=T_{i_1}^{\omega(\emptyset)}T_{i_2}^{\omega(i_1)}\cdots  
  T_{i_k}^{\omega(i_1\cdots i_{k-1})}.
\]
For all $\bi\in\Sigma^{\omega}$, define   
$Z^{\omega}(\bi)=\lim_{k\to\infty}f^\omega_{\bi_k}(0)$, where $0\in\mathbb R^D$,
and set $A^\omega =\{ Z^\omega(\bi)\mid \bi\in\Sigma^\omega\}$. We refer to
$A^\omega$ as the attractor or the code tree fractal corresponding to the 
code tree $\omega\in\Omega$. The attractor $A^\omega$ is well-defined since the
maps $f_i^\lambda$ are uniformly contracting.
For $k\in\mathbb N$, $\omega\in\Omega$ and $\bi\in\Sigma^\omega$, let 
\[
[\bi_k]=\{\bj\in\Sigma^\omega\mid j_l=i_l\text{ for all }l=1,\dots,k\}
\]
be the cylinder starting with $i_1\cdots i_k$. The length of the cylinder  
$[\bi_k]$ is $k$.

\begin{remark}\label{subset}

(a) One could define code tree fractals
for contractions on complete metric spaces, or more generally for maps such 
that the limit  $Z^\omega(\bi)$ exist for any $ \bi\in\Sigma^\omega$. However, 
in this paper we consider only affine contractions on $\mathbb R^D$.

(b) Any compact subset of a fractal generated by a single iterated function 
system is a code tree fractal. Indeed, let $A^W$ be the attractor of an
iterated function system $W=\{w_1,\dots,w_N\}$ on $\mathbb R^D$
and let $K\subset A^W$ be 
compact. Setting $I=\{1,\dots,N\}$, we denote by 
$Z^W\colon I^{\mathbb N}\to\mathbb R^D$ 
the natural projection and define $\widetilde K=(Z^W)^{-1}(K)$. For 
$k\in\mathbb N\cup\{0\}$, let
\[ 
\widetilde K_k=\{\bi_k\in I^k\mid\bi_k\bj\in\widetilde K\text{ for some }
  \bj\in I^{\mathbb N}\}.
\]
Further, let $\Lambda=\{\lambda\subset I\mid\lambda\ne\emptyset\}$ and 
${\bf F}=\{F^\lambda\mid\lambda\in\Lambda\}$, 
where
$F^\lambda=\{w_i\mid i\in\lambda\}$. Defining the mapping
$\omega\colon\bigcup_{k=0}^\infty\widetilde K_k\to\Lambda$ by
\[
\omega(i_1\cdots i_{k-1})=\{i\mid i_1\cdots i_{k-1}i\in\widetilde K_k\},
\]
we have $A^\omega=K$. In particular, sub-self-affine sets  are code tree 
fractals as compact subsets of self-affine sets. This in turn implies that 
attractors of graph directed Markov systems generated by affine maps are code 
tree fractals as well since they are sub-self-affine sets. 
See \cite{F95,F4,KV09,MU03}. 

It is also easy to see that a code tree fractal can be expressed as a subset of
the attractor of a single possibly infinite iterated function system.
\end{remark}

The way deterministic constructions are usually  randomized depends heavily on 
the formalism used. It is useful to add some additional
structure to our space in order to be able to compare results expressed in 
different formalisms. 

\begin{example}\label{eqrelation}
Suppose $\{w_1,w_2,w_3\}$ is an 
iterated function system, where $w_i=T_i+a_i$ for $i=1,2,3$. Consider the
code tree fractal constructed in Remark~\ref{subset} corresponding to a compact
$K\subset A^W$. Letting $\lambda=\{1,2\}$ and $\lambda'=\{2,3\}$, we have 
$F^\lambda=\{T_1^\lambda+a_1^\lambda,T_2^\lambda+a_2^\lambda\}$ and 
$F^{\lambda'}=\{T_1^{\lambda'}+a_1^{\lambda'},T_2^{\lambda'}+a_2^{\lambda'}\}$,
where $a_1^\lambda=a_1$, $a_2^\lambda=a_2$, $a_1^{\lambda'}=a_2$ and 
$a_2^{\lambda'}=a_3$. Thus the translation vector $a_2$ is represented both by
$a_2^\lambda$ and $a_1^{\lambda'}$. 
\end{example}

To allow identifications in our construction, we suppose that the set 
$\widehat\Lambda=\{(\lambda,i)\mid\lambda\in\Lambda\text{ and } 
i=1,\dots,M_\lambda\}$ is equipped with an equivalence relation $\sim$
satisfying the following two conditions:
\begin{itemize}
\item the cardinality $\mathcal A$ of the set of equivalence classes 
      $\ba:=\widehat\Lambda/\sim$ is finite,
\item for every $\lambda\in\Lambda$ we have $(\lambda,i)\sim (\lambda,j)$ if 
      and only if $i=j$. 
\end{itemize}
We use the notation $\ba$ for the set of equivalence classes to emphasize that 
we will identify only the translation parts of the maps by the equivalence 
relation $\sim$. The assumption $\mathcal A<\infty$ guarantees that the 
Lebesgue measure, $\mathcal L^{D\mathcal A}$, can be used
on the space $(\mathbb R^D)^{\mathcal A}$. By the second assumption,
within any system 
$\{T_1^\lambda+a_1^\lambda,\dots,T_{M_\lambda}^\lambda+a_{M_\lambda}^\lambda\}$
different translation vectors are never identified. Moreover, if 
$(\lambda,i)\sim (\lambda',j)$, then 
$a_i^\lambda=a_j^{\lambda'}$ as vectors in $\mathbb R^D$. Hence, we may 
consider $\ba$ as an element of $\mathbb R^{D\mathcal A}$, and we denote the 
corresponding attractor by $A_\ba^\omega$.

For a non-singular linear map $T\colon\mathbb R^D\to\mathbb R^D$, let 
$\sigma_i=\sigma_i(T)$, $1\leq i\leq D$, be the singular values of $T$ defined 
as the square roots of the eigenvalues of $T^*T$, where $T^*$ is the adjoint of
$T$. The singular values are enumerated in the decreasing order, that is,
\[
0<\sigma_D\leq\sigma_{D-1}\leq\dots\leq\sigma_2\leq\sigma_1=\Vert T\Vert. 
\]
Recall that the singular values are the lengths of the semi-axes  of the 
ellipsoid $T(B(0,1))$, where $B(x,\rho)\subset\mathbb R^D$ is the closed ball
with radius $\rho>0$ centred at $x\in\mathbb R^D$. Define the singular value 
function as follows
\[
\Phi^\alpha(T)
 =\begin{cases}\sigma_1\sigma_2\cdots\sigma_{m-1}\sigma_m^{\alpha-m+1},&
                \text{if } 0\le\alpha\le D,\\
   \sigma_1\sigma_2\cdots\sigma_{D-1} \sigma_D^{\alpha-D+1},&\text{if }\alpha>D,
  \end{cases}  
\]
where $m$ is the integer such that $m-1\le\alpha<m$. The singular value 
function $\Phi^\alpha$ is defined also for $\alpha>D$ to ensure the 
existence of the affinity dimension (see section~\ref{det}) for all affine code
tree fractals. We have 
$\sigma_D(T)^\alpha\le\Phi^\alpha(T)\le\sigma_1(T)^\alpha$, and in 
particular, $\Phi^\alpha(T)= s^\alpha$ if $T$ is a similitude with 
$\sigma_i=s$ for all $i=1,\dots,D$. The singular
value function is submultiplicative, that is,  
\[
\Phi^\alpha(TU)\le\Phi^\alpha(T)\Phi^\alpha(U)
\]
for all linear maps $T$ and $U$. For this and other properties of the singular 
value function see for example \cite{F88}. Throughout this paper we will 
make the assumption that there 
exist $\underline\sigma,\overline\sigma\in (0,1)$ such that
\[
0<\underline\sigma\le\sigma_D(T_i^\lambda)\le\sigma_1(T_i^\lambda)
  \le\overline\sigma<1
\]
for all $\lambda\in\Lambda$ and for all $i=1,\dots,M_\lambda$. Note that 
whilst the condition $\overline\sigma<1$ follows from the uniform 
contractivity assumption, the existence of $\underline\sigma>0$ is an extra
assumption since $\Lambda$ may be infinite even though $\mathcal A<\infty$.
The lower bound $\underline\sigma>0$ is not needed in section~\ref{det}.

\section{Deterministic  code tree fractals}\label{det} 
 
In this section we show that the result of Falconer \cite{F88} (sharpened 
by Solomyak \cite{S98}) concerning self-affine iterated function systems can
be extended to code tree fractals, that is, the Hausdorff dimension 
$\dimH(A_\ba^\omega)$ of a code tree fractal $A_\ba^\omega$ is typically
independent of the translation vector $\ba\in\mathbb R^{D\mathcal A}$. 
Theorem~\ref{0510} is not only an important ingredient in the proof of our main
theorem (Theorem~\ref{main}) but is also of independent interest. 
The original proofs by Falconer and Solomyak in the case $|\Lambda|=1$ 
generalize in a 
straightforward manner. For the convenience of the reader, we present the 
essential ideas.

We denote by $M^\alpha$ the $\alpha$-dimensional natural measure defined 
for all Borel subsets $E$ of $\Sigma^\omega$ by
\[
M^\alpha(E)=\lim_{j\to\infty}M_j^\alpha(E),
\]
where
\[
M_j^\alpha(E)=\inf\{\sum_{\bi_k\in J}\Phi^\alpha(T_{\bi_k}^\omega)\mid J\subset
  \Sigma_*^\omega, E\subset\bigcup_{\bi_k\in J}[\bi_k]\text{ and }k\ge j\}.
\]
The affinity dimension of $\Sigma^\omega$ is 
\[
d^\omega=\inf\{\alpha\mid M^\alpha(\Sigma^\omega)=0\} 
  =\sup\{\alpha\mid M^\alpha(\Sigma^\omega)=\infty\}.
\]  

Next lemma is the key tool in the proof of Theorem~\ref{0510}.
 
\begin{lemma}\label{nueng} 
Let $\rho>0$ and assume that 
$\overline\sigma<\tfrac 12$. If $\alpha$ is non-integral with $0<\alpha<D$ then
there exists a constant $c_1>0$ such that for any $k\in\mathbb N$ and 
$\bi,\bj\in\Sigma^\omega$ with $\bi_k=\bj_k$ and $i_{k+1}\neq j_{k+1}$, we have
\[
\int_{\ba\in B(0,\rho)}\frac{d\mathcal L^{D\mathcal A}(\ba)}
   {|Z_\ba^\omega(\bi)-Z_\ba^\omega(\bj)|^\alpha}
 \le\frac{c_1}{\Phi^\alpha(T_{\bi_k}^\omega)}\quad\text{for all }\omega\in\Omega.
\]
\end{lemma}

\begin{proof}
The points in $ A_\ba^\omega$ can be expressed as 
\begin{equation}\label{proj}
Z_\ba^\omega(\bi)=a_{i_1}^{\omega(\emptyset)}+ T_{i_1}^{\omega(\emptyset)}
  a_{i_2}^{\omega(i_1)}+T_{i_1}^{\omega(\emptyset)}T_{i_2}^{\omega(i_1)}
  a_{i_3}^{\omega(i_1i_2)}+\dotsb
\end{equation}
If $\bi,\bj\in [\bi_k]$, write $\bi=\bi_k\bi'$ and $\bj=\bi_k\bj'$ where
$i'_l=i_{k+l}$ and $j'_l=j_{k+l}$ for $l\in\mathbb N$. Now
\[
\int_{\ba\in B(0,\rho)}\frac{d\mathcal L^{D\mathcal A}(\ba)}
  {|Z_\ba^\omega(\bi)-Z_\ba^\omega(\bj)|^\alpha}=\int_{\ba\in B(0,\rho)}
  \frac{d\mathcal L^{D\mathcal A}(\ba)}{|T_{\bi_k}^\omega(Z_\ba^{\omega'}(\bi')
  -Z_\ba^{\omega'}(\bj'))|^\alpha}, 
\]
where the code tree $\omega'$ is defined by $\omega'(\bj_l)=\omega(\bi_k\bj_l)$
for $l\in\mathbb N$.  

Let
\[
n_1=\inf\{k\ge 2\mid (\omega'(\bi'_{k-1}),i'_k)\sim(\omega'(\bj'_{k-1}),j'_k)\}.
\]
We consider only the case $n_1<\infty$; the remaining case $n_1=\infty$ is 
easier (the second sum is not needed in \eqref{normestimate}). 
Letting $\ba=\{\beta_1,\dots,\beta_{\mathcal A}\}$, we may 
assume without loss of generality that 
$[(\omega'(\emptyset),i'_1)]=\beta_1$, $[(\omega'(\emptyset),j'_1)]=\beta_2$
and 
$[(\omega'(\bi'_{n_1-1}),i'_{n_1})]=[(\omega'(\bj'_{n_1-1}),j'_{n_1})]
  \ne\beta_1$. 
Define a linear map 
$E\colon\mathbb R^{D\mathcal A}\to\mathbb R^{D\mathcal A}$ by 
\[
E(\beta_1,\dots,\beta_{\mathcal A})=(y_1,\beta_2,\dots,\beta_{\mathcal A}),
\]
where
\[
y_1=Z_\ba^{\omega'}(\bi')-Z_\ba^{\omega'}(\bj')=\beta_1-\beta_2+
  \sum_{i=1}^{\mathcal A}L_i(\beta_i)
\]
and $L_i$ is a linear transformation on $\mathbb R^D$ 
for all $i=1,\dots,\mathcal A$
(recall \eqref{proj}). Further, since $\overline\sigma<\tfrac 12$, we have
\begin{equation}\label{normestimate}
\|L_1\|\le\sum_{k=1}^{n_1-2}\overline\sigma^k+\sum_{k=n_1}^\infty 
  2\overline\sigma^k=\frac{\overline\sigma-\overline\sigma^{n_1-1}}
  {1-\overline\sigma}+\frac{2\overline\sigma^{n_1}}{1-\overline\sigma} 
  <\frac{\overline\sigma}{1-\overline\sigma}<1.
\end{equation}
This implies that $\Id+L_1$ is 
inver\-tible (where the identity transformation on $\mathbb R^D$ is denoted by
$\Id$), 
and therefore, $E$ is invertible and $\det(E)=\det(\Id+L_1)\ge c_2>0$. 
By a change of variables, we 
obtain, using \cite[Lemma 2.2]{F88} in the last inequality, that
\begin{align*}
&\int_{\ba\in B(0,\rho)}\frac{d\mathcal L^{D\mathcal A}(\ba)}
  {|T_{\bi_k}^\omega(Z_\ba^{\omega'}(\bi')-Z_\ba^{\omega'}(\bj'))|^\alpha}
  = |\det(E)|^{-1}\int_{y\in E(B(0,\rho))}\frac{d\mathcal L^{D\mathcal A}(y)}
  {|T_{\bi_k}^\omega(y_1)|^\alpha}\\
 &\le c_2^{-1}\int_{B^D(0,\rho)}\dots\int_{B^D(0,\rho)}\int_{B^D(0,R)}
  \frac{d\mathcal L^D(y_1)\cdots d\mathcal L^D(y_{\mathcal A})}
  {|T_{\bi_k}^\omega(y_1)|^\alpha}\le\frac{c_1}{\Phi^\alpha(T_{\bi_k}^\omega)} 
\end{align*}
for some $0<R<\infty$ and $c_1<\infty$. Here the superscript $D$ emphasises 
that $B^D(0,r)$ is a closed ball in $\mathbb R^D$.
Note that $L_1$ and thus $E$ depend on $\bi,\bj$ and $\omega$, but $c_2$, $R$ 
and $c_1$ may be chosen to be independent of $\bi$, $\bj$ and $\omega$.
The assumption that $\alpha$ is non-integral with $0<\alpha<D$ is needed when 
applying \cite[Lemma 2.2]{F88}.
\end{proof}

Now we are ready to prove the main theorem of this section. The Hausdorff
dimension is denoted by $\dimH$.

\begin{theorem}\label{0510}
Let $\omega\in\Omega$ and assume that $\overline\sigma<\tfrac 12$. Then 
\[
\dimH(A_\ba^\omega)=\min\{D,d^\omega\}   
\]
for $\mathcal L^{D\mathcal A}$-almost all $\ba\in\mathbb R^{D\mathcal A}$.
\end{theorem}

\begin{proof}
As in \cite[Proposition 5.1]{F88} it follows that
\begin{equation}\label{upperbound}
\dimH(A_\ba^\omega)\le d^\omega
\end{equation}
for every $\ba\in\mathbb R^{D\mathcal A}$. Thus it suffices to prove that
$\dimH(A_\ba^\omega)\ge\min\{D,d^\omega\}$ for $\mathcal L^{D\mathcal A}$-almost
all $\ba\in\mathbb R^{D\mathcal A}$.
Let $\alpha$ be non-integral such that $0<\alpha<\min\{D,d^\omega\}$.  
As in \cite[Lemma 4.2]{F88} (see also \cite[Proposition 2.8]{RV}), there exists
a finite Borel measure $\mu^\omega$ on $\Sigma^\omega$ and a constant 
$c(\omega)$ such that 
\begin{equation}\label{0507}
\mu^\omega([\bi_k])\le c(\omega)\Phi^{\alpha}(T_{\bi_k}^\omega)
\end{equation} 
for any cylinder $[\bi_k]$.

Let $\rho>0$ and $s<\alpha$. From Lemma~\ref{nueng} and \eqref{0507} we get
\begin{align*}
\int_{\Sigma^\omega}&\int_{\Sigma^\omega}\int_{\ba\in B(0,\rho)}
   \frac{d\mathcal L^{D\mathcal A}(\ba)d\mu^\omega(\bi)d\mu^\omega(\bj)}
   {|Z_\ba^\omega(\bi)-Z_\ba^\omega(\bj)|^s}\\
 &\le\sum_{k=0}^\infty\sum_{\bi_k\in\Sigma_*^\omega}\Big(\sup_{\bj,\bj'\in [\bi_k],
   j_{k+1}\neq j'_{k+1}}\int_{\ba\in B(0,\rho)}\frac{\mu^\omega([\bj_k])
   \mu^\omega([\bj'_k])}{|Z_\ba^\omega(\bj)-Z_\ba^\omega(\bj')|^s}
   d\mathcal L^{D\mathcal A}(\ba)\Big)\\
 &\leq c_1\sum_{k=0}^\infty\sum_{\bi_k\in\Sigma_*^\omega}\frac{\mu^\omega([\bi_k])
   \mu^\omega([\bi_k])}{\Phi^s(T_{\bi_k}^\omega)}\le c(\omega)c_1\sum_{k=0}^\infty
   \sum_{\bi_k\in\Sigma_*^\omega}\frac{\mu^\omega([\bi_k])
   \Phi^\alpha(T_{\bi_k}^\omega)}{\Phi^s(T_{\bi_k}^\omega)}\\
 &\leq c(\omega)c_1\mu^\omega(\Sigma^\omega)\sum_{k=0}^\infty
   \overline\sigma^{k(\alpha-s)}<\infty.
\end{align*}
Finally, \cite[Lemma 5.2]{F88} (essentially Fubini's theorem and the mass 
distribution principle) implies $\dimH(A_\ba^\omega)\ge\alpha$ for 
$\mathcal L^{D\mathcal A}$-almost all $\ba\in\mathbb R^{D\mathcal A}$. The claim
follows by choosing a sequence $\alpha_i\uparrow\min\{D,d^\omega\}$.   
\end{proof}

It is common to relate the Hausdorff dimension to various notions of pressure. 
We will discuss this issue in the remaining part of this section. We start
by defining the natural pressure.
\begin{definition}\label{pressuredef}
For every $k\in\mathbb N$ and $\alpha\ge 0$, let
\[
S^\omega(k,\alpha)=\sum_{\bi_k\in\Sigma_*^\omega}\Phi^\alpha(T_{\bi_k}^\omega).
\]
Define 
\[
p_{\inf}^\omega(\alpha)=\liminf_{k\to\infty}\frac{\log S^\omega(k,\alpha)}k
  \text{ and }
p_{\sup}^\omega(\alpha)=\limsup_{k\to\infty}\frac{\log S^\omega(k,\alpha)}k. 
\]
If $p_{\inf}^\omega(\alpha)=p_{\sup}^\omega(\alpha)$, the common value is called 
the (natural) pressure $p^\omega(\alpha)$.
\end{definition}
As in the next section, it is easy to see that $p_{\inf}^\omega$ and 
$p_{\sup}^\omega$ are decreasing functions in $\alpha$ with uniquely defined 
$\alpha_0 \leq \alpha_1$ (depending on $\omega$) such that 
$p_{\inf}^\omega(\alpha_0)=p_{\sup}^\omega(\alpha_1)=0$. 

In \cite{F88} it is shown that in the case $|\Lambda|=1$ the natural pressure 
$p^\omega$ exists for self-affine iterated function systems
 and $d^\omega$ is the unique $\alpha$ such
that $p^\omega(\alpha)=0$, that is, the Hausdorff dimension of the attractor
is given by the unique zero of the pressure for 
$\mathcal L^{D\mathcal A}$-almost all $\ba\in\mathbb R^{D\mathcal A}$.

Below we 
construct examples illustrating various reasons why this result fails 
for general code tree fractals. All our examples are in $\mathbb R$ but 
they can easily be extended to $\mathbb R^D$ for any $D\in\mathbb N$.

\begin{example}\label{pressure1}
There exist code trees for which the pressure 
$p^\omega(\alpha)$ do not exist for any $\alpha>0$. 

For example, let $i=0,1$, and define $f_i,g_i:[0,1]\to[0,1]$ by 
$f_i(x)=\frac x8+\frac 78i$ and $g_i(x)=\frac x4+\frac 34i$ for all 
$x\in[0,1]$. Let $F^1=\{f_0,f_1\}$ and $F^2=\{g_0,g_1\}$. The equivalence 
relation $\sim$ is trivial, that is, $\ba\in\mathbb R^4$. Letting 
$1=N_0<N_1<N_2<\ldots$ be integers, set $\omega(\emptyset)=2$ and for all 
$l=0,1,2,\dots$
\begin{align*}
\omega(\bi_k)=\begin{cases}
             1,&\text{ if }N_{2l}\le k<N_{2l+1},\\
             2,&\text{ if }N_{2l+1}\le k<N_{2(l+1)}.
\end{cases}
\end{align*}
It is easy to see that for a sufficiently rapidly increasing  
sequence $(N_i)$ we obtain
\[
p_{\inf}^\omega(\alpha)=(1-3\alpha)\log2\text{ and }
p_{\sup}^\omega(\alpha)=(1-2\alpha)\log2.
\]
Observe that for an open set of translation vectors $\ba$ we have
\[
\udimb(A_\ba^\omega)=\dimp(A_\ba^\omega)=\frac 12\text{ and } 
\ldimb(A_\ba^\omega)=\dimH(A_\ba^\omega)=\frac 13,
\]
where $\ldimb$ and $\udimb$ are the lower and upper box 
counting dimensions, respectively, and $\dimp$ is the packing dimension.
\end{example} 

The following example shows that the Hausdorff dimension of a code tree 
fractal may be strictly smaller than the unique zero of the pressure for
all translation vectors.

\begin{example}\label{pressure2}
We construct a code tree $\omega$ for which the pressure 
$p^\omega(\alpha)$ exists for all $\alpha\ge 0$ and there exists a unique 
$d$ such that $p^\omega(d)=0$. However, $\dimH(A_\ba^\omega)<d$ for all 
translation vectors $\ba$.

Let $0<r<R\le\frac 13$. We consider two systems $F=\{f_1, f_2, f_3\}$ and
$G=\{g_1, g_2, g_3\}$ consisting of similarities on $\mathbb R$ such that 
$f_i(x)=rx+a_i^1$ and $g_i(x)=Rx+a_i^2$ for all $i=1,2,3$ and $x\in\mathbb R$
with the trivial equivalence relation $\sim$, that is, $\ba\in\mathbb R^6$.
Taking a sequence of integers 
$1=N_0<N_1<N_2<\ldots$, define a code tree $\omega$ as follows:
for each $n\in\mathbb N\cup\{0\}$, $m\in\{1,2,3\}$ and 
$\tau\in\bigcup_{k=0}^\infty\{1,2,3\}^k$ for which 
$N_{3n+(m-1)}\le|\tau|<N_{3n+m}$, set
\begin{align*}
\omega(\tau)&=\begin{cases}
               F,&\text{ if } \tau_1=m,\\
               G &\text{ otherwise.}
\end{cases}
\end{align*}
If $N_k\to\infty$ fast enough, we obtain for all $\alpha\ge0$
\[ 
p^\omega(\alpha)=\lim_{k\to\infty}\frac{\log S^\omega(k,\alpha)}k
  =\lim_{k\to\infty}\frac{\log(3^kR^{k\alpha})}k=\log 3+\alpha\log R,
\]
giving $p^\omega(\alpha)=0$ if and only if $\alpha=-\log 3/\log R$. On
the other hand, it is straightforward to see that
$M^\alpha(\Sigma^\omega)<\infty$ if $\alpha>-\log 3/\log r$. This implies
by \eqref{upperbound} that
\[
\dimH(A_\ba^\omega)\le d^\omega\le-\log 3/\log r<-\log 3/\log R. 
\]
Furthermore,
$\dimp(A_\ba^\omega)=-\log 3/\log R$ in an open set of translation vectors 
$\ba\in\mathbb R^6$.

Observe that 
$p^\omega(\alpha)$ does not exist if
we restrict $\Sigma^\omega$ to any of the branches 
$\Sigma_m^\omega=\{\bi\in\Sigma^\omega\mid i_1=m\}$ for $m\in\{1,2,3\}$. It is 
possible to modify the construction in such a way that
the proportion of nodes where $\omega(\bi_k)=F$ decreases to zero as $k$ tends
to infinity but for every $\bi$ there are infinitely many $l$ for which 
$\omega(\bi_k)=F$ for all $N_l\le k<N_{l+1}$. 
Then for all $\alpha\ge 0$, we have 
$p^\omega(\alpha)=\log3+\alpha\log R$, and this
remains true also for all branches of $\Sigma^\omega$. The modification does not
affect the Hausdorff dimension of the code tree fractal. 
\end{example}

In the previous example the packing dimension of the code tree fractal is 
equal to
the unique zero of the pressure in an open set of translation vectors.
The last example of this section indicates that this is not always the case.
 
\begin{example}\label{pressure3}
There exist a code tree $\omega$ such that the pressure 
$p^\omega(\alpha)$ exists for all $\alpha\ge0$ and for all translation
vectors $\ba$ neither
$\dimH(A_\ba^\omega)$ nor $\dimp(A_\ba^\omega)$ agrees with
the unique zero of the pressure.

For $i=0,1$, define $f_i:[0,1]\to[0,1]$ by $f_i(x)=\frac x2+a_i$ for
all $x\in[0,1]$, where $a_0=0$ and $a_1=\tfrac 12$. Set $F^1=\{f_0,f_1\}$, 
$F^2=\{f_0\}$ and $F^3=\{f_1\}$. We identify $a_0$ in $F^2$ with $a_0$ in $F^1$
and  $a_1$ in $F^3$ with $a_1$ in $F^1$.
Let $\omega$ be such that the corresponding attractor is 
$A_\ba^\omega=\{0\}\cup\{\frac 1n\mid n\in\mathbb N\}$ 
(see Remark~\ref{subset}).
Clearly, $\dimH(A_\ba^\omega)=\dimp(A_\ba^\omega)=0$ and 
$\dimb(A_\ba^\omega)=\frac 12$. 
Moreover, denoting by $N(k)$ the number of dyadic intervals of length $2^{-k}$
that meet $A_\ba^\omega$, we obtain for all $\alpha\ge0$
\[
p^\omega(\alpha)=-\alpha\log2+\lim_{k\to\infty}\frac{\log N(k)}k,
\]
implying $p^\omega(\alpha)=0$ if and only 
if $\alpha=\dimb(A_\ba^\omega)=\frac 12$. Let us now
 consider translations of this system.
Since the attractor $A_\ba^\omega$ is a countable set
for all $\ba\in\mathbb R^2$, we have $\dimp(A^\omega_\ba)=0$ for all
$\ba\in\mathbb R^2$. Note that the zero of the pressure $p^\omega$ is 
independent of $\ba\in\mathbb R^2$.   
\end{example}
 
We finish this section with a proposition concerning 1-variable code trees
(for the definition see
Example \ref{assumpsat}).

\begin{proposition}\label{V259}
Suppose that $\omega$ is a 1-variable (homogeneous) code tree and ${\bf F}$ 
contains only similarities. Then 
$d^\omega$ equals the unique $\alpha_0$ such that $p_{\inf}^\omega(\alpha_0)=0$.
\end{proposition}

\begin{proof}
Letting $\beta>\alpha_0$, there is $c<1$ satisfying
\[
\liminf_{k\rightarrow\infty}\left(S^\omega(k,\beta)\right)^{1/k}<c.
\] 
Consequently, for arbitrarily large values of $k$, we get 
$\sum_{\bi_k\in\Sigma_*^\omega}\Phi^\beta(T_{\bi_k}^\omega)<c^k$ which implies
that $M^\beta(\Sigma^{\omega})=0$. Letting $\beta\downarrow\alpha_0$, we get
$d^\omega\leq\alpha_0$.

To prove the opposite inequality, consider $\beta>d^\omega$. Then
$M^\beta(\Sigma^{\omega})=0$. Given $\varepsilon>0$, we find an index set
$J\subset\Sigma^{\omega}_*$ such that
$\Sigma^{\omega}\subset\bigcup_{\tau\in J}[\tau]$ and 
$\sum_{\tau\in J}\Phi^\beta(T^\omega_\tau)<\varepsilon$. 
Since $\Sigma^\omega$ is compact, we may assume that $J$ is finite.
Defining 
\[
n_{\min}=\min\{|\tau|\,\mid\,\tau\in J\}\text{ and }
n_{\max}=\max\{|\tau|\,\mid\,\tau\in J\},
\] 
we may assume that $J$ minimizes
\begin{equation}\label{eq:min}
\sum_{\tau\in K}\Phi^\beta(T^\omega_\tau)
\end{equation}
among all covering sets $K$ that satisfy $n_{\min}\leq|\tau|\leq n_{\max}$ for 
all $\tau\in K$. Consider $\zeta\in\Sigma_{n_{\min}}^\omega$ and let 
$J_\zeta=\{\tau\in J\,\mid\,[\tau]\subset[\zeta]\}$. Since $J$ minimizes 
\eqref{eq:min} and each map $f^{\lambda}_i$ is similarity, it follows from the 
homogeneity of $\omega$ that 
$\sum_{\tau\in J_\zeta}\Phi^\beta(T^\omega_\tau)=\Phi^\beta(T^\omega_\zeta)$. 
Consequently, we obtain 
\[
S^\omega(n_{\min},\beta)=\sum_{\tau\in J}\Phi^\beta(T^\omega_\tau)
<\varepsilon.
\]
Choosing sufficiently large $n_{\min}\in\mathbb N$ and sufficiently small
$\varepsilon>0$, it
follows that $\liminf_{k\rightarrow\infty}S^\omega(k,\beta)=0$. This 
yields
$p^\omega_{\inf}(\beta)\leq 0$, and therefore, $\alpha_0\leq\beta$. Letting 
$\beta\downarrow d^\omega$, we get $\alpha_0\leq d^\omega$.
\end{proof}

\begin{remark}\label{UOSC}
There is not much to be said about the dimension of a code tree fractal 
for a fixed
translation parameter $\ba\in\mathbb R^{D\mathcal A}$. The following can be 
gleaned from the proof of
\cite[Theorem 7.3]{F86}: Let $\omega\in\Omega$. Assume that
all the maps are similarities and for some fixed 
$\ba\in\mathbb R^{D\mathcal A}$ the following uniform open set condition is 
satisfied: there is a non-empty open set $O$ such that for each  
$\lambda\in\Lambda$ we have 
\begin{equation}\label{UOSC2} 
\bigcup_{m=1}^{M_\lambda}f_m^\lambda(O)\subset O\text{ and }
  f_m^\lambda(O)\cap f_n^\lambda(O)=\emptyset\text{ if }m\neq n.
\end{equation} 
Then $\dimH(A_\ba^\omega)=d^\omega$.
\end{remark}

\section{Random code tree fractals and pressure}\label{random} 

As illustrated in the previous section, in general, the dimension of a code
tree fractal is not necessarily given by the zero of the natural pressure. In 
this section we consider a general class 
of random affine code tree fractals for which the pressure almost surely
exists, is independent of the code tree and has a unique zero. 
When constructing the random sets in \cite{FM09} and \cite{JPS}, independent
choices are made at different steps of the construction and the limiting sets 
have a stochastically self-repeating structure both in space and in scale
whilst in our model the random sets are spatially nearly homogeneous.
More precisely, suppose that $(\mu_k)_{k=0}^\infty$ is a family 
of probability measures, where each $\mu_k$ is a measure on the set of finite 
code trees of length $k$. Suppose also that $\nu$ is a probability measure on 
$\mathbb N$ with finite first moment. This generates a probability measure $P$ 
on $\Omega$ in the 
following way: Choose a random number $N_1$ according to $\nu$. Select labels 
of nodes from $\Lambda$ from level $0$ up to level $N_1-1$ at random according 
to $\mu_{N_1-1}$. This specifies a randomly chosen code tree up to level 
$N_1-1$. Repeat 
this procedure, that is, generate a realization of $N_2-N_1$ according to 
$\nu$. For some fixed subtree rooted at level $N_1$ choose labels of nodes of 
this subtree at random according to the probability measure $\mu_{N_2-N_1-1}$. 
For all other subtrees rooted at level $N_1$ select the same labels as in 
the generated one. This specifies the randomly chosen code tree up to level 
$N_2-1$. Continuing in this manner will uniquely define $P$ according to 
Kolmogorov's extension theorem. 

Next we present a convenient formalism to study constructions described above. 

\begin{definition}\label{space}
Let  $\widetilde\Omega$ be the set of
$(\omega,N)\in\Omega\times\mathbb N^{\mathbb N}$ such that
\begin{itemize}
\item $N_m<N_{m+1}$ for all $m\in\mathbb N$,
\item if $\bi_{N_m}\bj_l,\bi'_{N_m}\in\Sigma_*^\omega$, then 
$\bi'_{N_m}\bj_l\in\Sigma_*^\omega$ and 
$\omega(\bi_{N_m}\bj_l)=\omega(\bi'_{N_m}\bj_l)$,
\end{itemize}
where for each $(\omega,N)\in\widetilde\Omega$ the sequence 
$(N_m)_{m\in\mathbb N}$ is the list of \textbf{neck levels}, that is, all the 
sub code trees rooted at a neck level $N_m$ are identical. 
Define $\Xi\colon\widetilde\Omega\to\widetilde\Omega$ by 
$\Xi(\omega,N)=(\hat\omega,\hat N)$, where 
$\hat N_m=N_{m+1}-N_1$ and $\hat\omega(\bj_l)=\omega(\bi_{N_1}\bj_l)$
for all $m,l\in\mathbb N$. The elements of $\widetilde\Omega$ are
denoted by $\tilde\omega$. We equip $\widetilde\Omega$ with the topology 
generated by the cylinders
\begin{align*}
[(\omega,N)_m]=\{(\hat\omega,\hat N)\in\widetilde\Omega &\mid
  \hat N_i=N_i\text{ for all }i\le m\text{ and }\hat\omega(\tau)=\omega(\tau)\\
&\text{ for all }\tau\text{ with } |\tau|<N_m\}
\end{align*}
and use the Borel $\sigma$-algebra on $\widetilde\Omega$.
\end{definition}

\begin{remark}\label{topology} 
(a) Observe that since $N_1$ is a neck level, the definition of $\Xi$ is 
independent of the choice of $\bi_{N_1}$. With the chosen
topology, $\Xi$ is continuous. Since there is no uniform upper bound 
for $N_1$, the space $\widetilde\Omega$ is not compact.

(b) The functions $p^\omega$, $\Phi^\alpha(T^\omega)$ etc. defined  
in the previous sections have natural extensions to $\widetilde\Omega$, that 
is, $p^{(\omega,N)}=p^\omega$, $\Phi^\alpha(T^{(\omega,N)})=\Phi^\alpha(T^\omega)$ 
etc. We denote them by $p^{\tilde\omega}$, $\Phi^\alpha(T^{\tilde\omega})$ etc. 
\end{remark} 

We complete this section by proving the existence of the zero of the 
pressure.

\begin{theorem}\label{pexists}
Let $0<\underline\sigma\le\overline\sigma<1$ be as in section~\ref{preli}.
Assume that $P$ is an ergodic $\Xi$-invariant probability measure on 
$\widetilde\Omega$ such that 
$\int_{\widetilde\Omega}N_1(\tilde\omega)\,dP(\tilde\omega)<\infty$. 
Then for $P$-almost all $\tilde\omega\in\widetilde\Omega$ the pressure 
$p^{\tilde\omega}(\alpha)$ exists for all $\alpha\in[0,\infty[$ and is 
independent of $\tilde\omega\in\widetilde\Omega$. 
Furthermore, $p^{\tilde\omega}$ is strictly decreasing and there 
exists a unique $\alpha_0$ such that $p^{\tilde\omega}(\alpha_0)=0$ for 
$P$-almost all $\tilde\omega\in\widetilde\Omega$. 
\end{theorem}

\begin{proof} 
Consider $\alpha\in[0,\infty[$. For $n<m\in\mathbb N\cup\{0\}$, let
\[
\Sigma_*^{\tilde\omega}(n,m)=\{i_{N_n+1}\cdots i_{N_m}\mid\bi_{N_n}i_{N_n+1}\cdots
  i_{N_m}\in\Sigma_*^{\tilde\omega}\}, 
\]
where $N_0=0$. The fact that $N_n$ is a neck level implies that
the definition is independent of the choice of 
$\bi_{N_n}$. Setting
\[
 X_{n,m}(\tilde\omega)=\log\bigl(\sum_{i_{N_n+1}\cdots i_{N_m}
  \in\Sigma_*^{\tilde\omega}(n,m)}
  \Phi^\alpha(T_{i_{N_n+1}}^{\tilde\omega(i_1\cdots i_{N_n})}\cdots 
  T_{i_{N_m}}^{\tilde\omega(i_1\cdots i_{N_m-1})})\bigr),
\]
we have 
\[
X_{n+1,m+1}=X_{n,m}\circ\Xi
\]
by the definition of $\Xi$. Note that
$X_{0,n}(\tilde\omega)=\log(S^{\tilde\omega}(N_n,\alpha))$,
and moreover, submultiplicativity of $\Phi^\alpha$ gives
\begin{equation}\label{subadd}
X_{0,m}\le X_{0,n}+ X_{n,m}
\end{equation}  
for any $0<n<m$. Since 
\[
(N_{n+1}-N_n)\log(\underline\sigma^\alpha)\le X_{n,n+1}\le (N_{n+1}-N_n)
\log(\overline\sigma^\alpha M)
\]
and $N_1\circ\Xi^n=N_{n+1}-N_n$, combining the $\Xi$-invariance of $P$ and the 
assumption $\int_{\widetilde\Omega}N_1(\tilde\omega)\,dP(\tilde\omega)<\infty$,
gives that $X_{n,m}$ is $P$-integrable for all $n,m\in\mathbb N\cup\{0\}$ with 
$n<m$. From Kingman's subadditive ergodic theorem (see for example Durrett 
\cite{D91}) it follows that the limit
\[
\lim_{n\to\infty}\frac{\log(S^{\tilde\omega}(N_n,\alpha))}{n}
  =:\tilde p^{\tilde\omega}(\alpha)
\]
exists for $P$-almost all $\tilde\omega\in\widetilde\Omega$.
Furthermore, $\tilde p^{\tilde\omega}$ is 
$P$-almost surely independent of $\tilde\omega$ by ergodicity of $P$. 
Observe that in the definition of $\tilde p^{\tilde\omega}$ there is a neck level
$N_n$ in the numerator whilst in the denominator we have $n$.
Birkhoff's ergodic theorem implies that the finite non-random limit
\[
\lim_{n\to\infty}\frac{N_n}n=\lim_{n\to\infty}\frac 1n
  \sum_{k=0}^{n-1}N_1\circ\Xi^k=b\ge 1
\]
exists for $P$-almost all $\tilde\omega\in\widetilde\Omega$. In particular, 
\begin{equation}\label{smallstep}
\lim_{n\to\infty}\frac{N_{n+1}-N_n}n=0=\lim_{n\to\infty}\frac{N_{n+1}-N_n}{N_n},
\end{equation} 
and therefore, the limit 
\[
\lim_{n\to\infty}\frac{\log(S^{\tilde\omega}(N_n,\alpha))}{N_n}
  =\frac{\tilde p^{\tilde\omega}(\alpha)}{b}
\]
exists for $P$-almost all $\tilde\omega\in\widetilde\Omega$.

For any $m\in\mathbb N$, let $l(m)$ be a random integer such that 
$N_{l(m)}\leq m< N_{l(m)+1}$. Since $l(m)\le m$ and
\[
(\underline\sigma^\alpha)^{N_{l(m)+1}- N_{l(m)}}S^{\tilde\omega}(N_{l(m)},\alpha)
  \leq S^{\tilde\omega}(m,\alpha)
  \leq (\overline\sigma^\alpha M)^{N_{l(m)+1}-N_{l(m)}}
    S^{\tilde\omega}(N_{l(m)},\alpha),
\]
equation \eqref{smallstep} implies that the pressure
\begin{equation}\label{uniquepressure}
\lim_{m\to\infty}\frac{\log(S^{\tilde\omega}(m,\alpha))}m
  =p^{\tilde\omega}(\alpha)=\frac{\tilde p^{\tilde\omega}(\alpha)}{b}
\end{equation}
exists for $P$-almost all $\tilde\omega\in\widetilde\Omega$. By ergodicity, 
$p^{\tilde\omega}(\alpha)$ is  $P$-almost surely independent of 
$\tilde\omega\in\widetilde\Omega$. 

We continue by verifying that
for $P$-almost all $\tilde\omega\in\widetilde\Omega$ the pressure 
$p^{\tilde\omega}(\alpha)$ exists for all $\alpha\in[0,\infty[$, it is strictly 
decreasing in $\alpha$ and it has a unique zero at $\alpha_0$.
It follows from Fubini's theorem that for 
$P$-almost all $\tilde\omega\in\widetilde\Omega$ the pressure 
$p^{\tilde\omega}(\alpha)$ exists for $\mathcal L$-almost all 
$\alpha\in [0,\infty[$. Consider such $\tilde\omega\in\widetilde\Omega$. 
Using the
properties of the singular value function, we obtain for all 
$\alpha\in [0,\infty[$, for all $\delta>0$ and for all $m\in\mathbb N$
\[
\underline\sigma^{m\delta}\leq\frac{S^{\tilde\omega}(m,\alpha+\delta)}
  {S^{\tilde\omega}(m,\alpha)}\leq\overline\sigma^{m\delta}.
\]
Hence both
\[
\liminf_{n\to\infty}\frac{\log S^{\tilde\omega}(n,\alpha)}{n}\, \text{ and }\,
 \limsup_{n\to\infty}\frac{\log S^{\tilde\omega}(n,\alpha)}{n}
\]
are strictly decreasing and continuous in $\alpha$. Since for 
$\mathcal L$-almost all $\alpha\in[0,\infty[$ the pressure
$p^{\tilde\omega}(\alpha)$ exists, it follows 
that the limit exists for all $\alpha\in[0,\infty[$ and is continuous and  
strictly decreasing in $\alpha$. Finally, the fact that for all $n\in\mathbb N$ 
\[
\frac{\log S^{\tilde\omega}(n,0)}n\geq 0\, \text{ and }
 \lim_{\alpha\to\infty}\lim_{n\to\infty}\frac{\log S^{\tilde\omega}(n,\alpha)}n
 =-\infty 
\]
yields the existence of a unique $\alpha_0$ such that 
$p^{\tilde\omega}(\alpha_0)=0$. This completes the proof. 
\end{proof}

\begin{remark}\label{Julgubben}
By \eqref{uniquepressure}, the functions $\tilde p^{\tilde\omega}$ and
$p^{\tilde\omega}$ have the same unique zero for $P$-almost all 
$\tilde\omega\in\widetilde\Omega$.
\end{remark}

\begin{example}\label{assumpsat}
(a) The definition of $\widetilde\Omega$ was inspired by the construction of
random V-variable fractals introduced by Barnsley et al.\ (see 
\cite{BHS2005,BHS2008,BHS08}). Let $V \geq 1$ be an integer.
A code tree is said to be $V$-variable if there are at most $V$ 
distinct sub code trees at each level, and a code tree fractal is $V$-variable 
if the 
code tree is $V$-variable. The attractors of ordinary iterated function systems
are $1$-variable fractals also known as homogeneous fractals. In random 
V-variable code trees the distribution of the spacing between necks has
exponentially decreasing tails implying that $N_1$ has a finite expectation.
In particular, the probability measure related to random V-variable fractals 
introduced by Barnsley et al. satisfies the assumptions of 
Theorem~\ref{pexists}.

It was shown in \cite{BHS08} that, under the assumption that all maps are 
similarities, 
$S^{\tilde\omega}(n,\alpha)$ can be expressed in terms of products of
random $V\times V$-matrices and, using the law of large numbers, one obtains
\[
p^{\tilde\omega}(\alpha)=\int_{\widetilde\Omega}\log\bigl(\sum_{\bi_{N_1}\in
  \Sigma_*^{\tilde\omega}
  (0,1)}(r_{i_1}^{\tilde\omega(\emptyset)}r_{i_2}^{\tilde\omega(i_1)}\cdots 
  r_{i_{N_1}}^{\tilde\omega(i_1\cdots i_{{N_1}-1})})^\alpha\bigr)\,dP(\tilde\omega),
\]
where $r_i^\lambda$ is the similarity ratio of $f_i^\lambda$. This
property is useful when estimating the 
pressure and the dimension numerically. 
In general, it is hard to give good numerical 
estimates for the pressure when 
$\int_{\widetilde\Omega}N_1(\tilde\omega)\,dP(\tilde\omega)$ is large.

(b) Another example of systems satisfying the assumptions of 
Theorem~\ref{pexists} are random Markov homogeneous fractals.
Suppose that $\Lambda=\{1,\dots,\ell\}$ is finite and $Q$ is an ergodic Markov
transition matrix on $\Lambda$, that is, $Q$ is an $\ell\times\ell$-matrix 
with non-negative elements and row-sums equal to 1. Let $P_0$ be a probability 
measure on $\Lambda$. 

A probability measure on the set of code trees 
$\Omega$ is generated in the following way: Suppose 
$\omega(\emptyset)=\omega_0$ is chosen 
according to $P_0$. Next choose $\omega_1$ according to the probability 
measure defined by the row $\omega_0$ in $Q$, that is, 
$\omega_1=j$ with probability $Q_{\omega_0 j}$. Further, let 
$\omega(i)=\omega_1$ for all $i=1,\dots,M_{\omega(\emptyset)}$. This defines a 
measure on $\Omega$ up to level 1. We continue inductively in the same 
manner, that is, $\omega_{k+1}=j$ with probability $Q_{\omega_k j}$ and
$\omega(i_1\cdots i_kl)=\omega_{k+1}$ for all 
$l=1,\dots,M_{\omega(i_1\dots i_k)}$. This defines a probability measure $P$ on 
$\Omega$. 

Every realization of the above construction is spatially homogeneous since
at every level all the sub code trees are identical. By ergodicity $P$ has a 
self-repeating time structure in the following sense: Let $N_i(\omega)$ be the 
smallest level $k>N_{i-1}(\omega)$ such that $\omega_{N_i(\omega)}=\omega_0$, 
where $N_0(\omega)=0$. The sequence $(N_i)_{i\in\mathbb N}$ will define  
regenerative time points in the sense that the probability measure $P$ on 
$\Omega$ will have self-repeating structure at these random times. 
\end{example}

\section{Dimensions of random code tree fractals}\label{dimension}

In this section we show that for $P$-almost all 
$\tilde\omega\in\widetilde\Omega$ the zero of 
the pressure is equal to
the Hausdorff dimension of 
$A_\ba^{\tilde\omega}$ for $\mathcal L^{D\mathcal A}$-almost all 
$\ba\in\mathbb R^{D\mathcal A}$. For this purpose, we restrict our consideration
to the case $D=2$ and impose additional assumptions on the probability measure 
$P$. For a discussion concerning the assumptions of the following theorem,
see Remark~\ref{assumptions}. 

\begin{theorem}\label{main}
Let $D=2$ and $0<\underline\sigma\le\overline\sigma<\frac 12$. Assume that
$P$ is an ergodic $\Xi$-invariant probability measure on $\widetilde\Omega$.
Further, suppose that 
\begin{equation}\label{direction}
\begin{aligned}
P\bigl(\{\tilde\omega\in\widetilde\Omega \mid & \text{ there exists }
 v\in\mathbb R^2\setminus\{0\}\text{ such that }T_{\bi_{N_1}}^{\tilde\omega}(v)\\
&\text{ are parallel for all }\bi_{N_1}\in \Sigma_*^{\tilde\omega}(0,1)\}\bigr)
  <1,
\end{aligned}
\end{equation}
\begin{equation}\label{decay}
\sum_{l=1}^\infty P\bigl(\{\tilde\omega\in\widetilde\Omega\mid N_1(\tilde\omega)
  \ge\frac{c_0l}{\log l}\}\bigr)<\infty\text{ for all }c_0>0
\end{equation}
and the $\sigma$-algebras generated by $\{\tilde\omega(\tau)\mid |\tau|<N_m\}$
and  $\{\tilde\omega(\tau)\mid |\tau|\ge N_m\}$ are independent for all 
$m\in\mathbb N$. Then for $P$-almost all $\tilde\omega\in\widetilde\Omega$
\[
\dimH(A_{\ba}^{\tilde\omega})=\min\{\alpha_0,2\}
\]
for $\mathcal L^{2\mathcal A}$-almost all $\ba\in\mathbb R^{2\mathcal A}$, where
$\alpha_0$ is the zero of the pressure given in Theorem~\ref{pexists}.
\end{theorem}

\begin{proof}
Theorem~\ref{main} is proved as a consequence of a sequence of lemmas.
Combining Theorem~\ref{0510} with the fact that $d^{\tilde\omega}\le\alpha_0$
for $P$-almost all $\tilde\omega\in\widetilde\Omega$,
implies $\dimH(A_\ba^{\tilde\omega})\le\alpha_0$ and therefore, it suffices to 
verify that for $P$-almost all $\tilde\omega\in\widetilde\Omega$,
$\alpha_0\le\dimH(A_\ba^{\tilde\omega})$ for $\mathcal L^{2\mathcal A}$-almost all
$\ba\in\mathbb R^{2\mathcal A}$. For this purpose we
construct for all $\alpha<\alpha_0$ and for $P$-almost all 
$\tilde\omega\in\widetilde\Omega$ a probability measure $\mu^{\tilde\omega}$ 
on $\Sigma^{\tilde\omega}$ such that for some constant $c(\tilde\omega)$
\begin{equation}\label{eq66}
\mu^{\tilde\omega}([\bi_l])
\leq c(\tilde\omega)\Phi^\alpha(T_{\bi_l}^{\tilde\omega}) 
\end{equation}
for all $l\in\mathbb N$ (recall \eqref{0507} and the proof of 
Theorem~\ref{0510}). Observe that inequality \eqref{0507} is valid only for 
$\alpha<d^{\tilde\omega}$.
For $\tilde\omega\in\widetilde\Omega$ and $\alpha<\alpha_0$ the 
measure $\mu^{\tilde\omega}$ on $\Sigma^{\tilde\omega}$ is defined 
in the following way: Let $m\in\mathbb N$ and set
\begin{equation}\label{muomega}
\mu_m^{\tilde\omega} = \frac{\sum_{\bi_{N_m}\in\Sigma_*^{\tilde\omega}(0,m)}
   \Phi^\alpha(T_{\bi_{N_m}}^{\tilde\omega})\delta_{\bi_{N_m}}}{\sum_{\bi_{N_m}
   \in\Sigma_*^{\tilde\omega}(0,m)}
   \Phi^\alpha(T_{\bi_{N_m}}^{\tilde\omega})},
\end{equation}
where $\delta_{\bi_{N_m}}$ is the Dirac measure at a fixed point of the cylinder 
$[\bi_{N_m}]$. The choice of the cylinder point plays no role in what follows.
Since $\Sigma^{\tilde\omega}$ is compact, the sequence 
$(\mu_m^{\tilde\omega})_{m\in\mathbb N}$ has a converging 
subsequence with a limit measure $\mu^{\tilde\omega}$.
Observe that the converging subsequence 
may depend on $\tilde\omega\in\widetilde\Omega$.

By submultiplicativity of $\Phi^\alpha$, we have for all $n,m\in\mathbb N$ and
$N_{n-1}\le l<N_n$ that
\begin{equation}\label{measurebound}
\mu_{n+m}^{\tilde\omega}([\bi_l])
  \le\frac{\Phi^\alpha(T_{\bi_l}^{\tilde\omega})\sum_{\bj,\bi_l\bj\in
  \Sigma_*^{\tilde\omega}(0,n)}\sum_{\bk_{N_m}\in\Sigma_*^{\tilde\omega}
  (n,n+m)}\Phi^\alpha(T_{(\bi_l)\bj}^{\tilde\omega})
  \Phi^\alpha(T_{\bk_{N_m}}^{\Xi^n(\tilde\omega)})}   
  {\sum_{\bi_{N_n}\in\Sigma_*^{\tilde\omega}(0,n)}
  \sum_{\bk_{N_m}\in\Sigma_*^{\tilde\omega}(n,n+m)}\Phi^\alpha
  (T_{\bi_{N_n}}^{\tilde\omega}T_{\bk_{N_m}}^{\Xi^n(\tilde\omega)})},
\end{equation}
where the last $|\bj|$ maps of $T_{\bi_l\bj}^{\tilde\omega}$ are denoted by
$T_{(\bi_l)\bj}^{\tilde\omega}$. In order to prove \eqref{eq66},
we need to estimate the denominator of \eqref{measurebound} from below by 
\begin{equation}\label{denominator}
\tilde c(\tilde\omega)\sum_{\bi_{N_n}\in\Sigma_*^{\tilde\omega}(0,n)}
  \sum_{\bk_{N_m}\in\Sigma_*^{\tilde\omega}(n,n+m)}\Phi^\alpha
  (T_{\bi_{N_n}}^{\tilde\omega})\Phi^\alpha(T_{\bk_{N_m}}^{\Xi^n(\tilde\omega)})
\end{equation}
for some $\tilde c(\tilde\omega)>0$. 

The verification of the lower bound \eqref{denominator} is divided into three
different steps. The first two steps are Lemmas~\ref{poslowpressure} 
and \ref{notaligned} in which the assumptions \eqref{direction}, 
\eqref{decay} and the independence are not needed. However, the assumptions of 
Theorem~\ref{pexists} have to be valid. 

Since $D=2$, it is enough to consider
$0\le\alpha\le 2$. Letting $T,U:\mathbb R^2\to\mathbb R^2$ be linear maps, set
\[
\underline\Phi^\alpha(T)=
 \begin{cases}\sigma_2(T)^\alpha,&\text{ if }\alpha\le 1\\
              \sigma_2(T)\sigma_1(T)^{\alpha-1},&\text{ if }\alpha>1.
 \end{cases}
\]

\begin{lemma}\label{philowerbound}
Let $T,U:\mathbb R^2\to\mathbb R^2$ be  linear maps and $0\le\alpha\le 2$. Then
\[
\Phi^\alpha(TU)\ge\underline\Phi^\alpha(T)\Phi^\alpha(U).
\]
\end{lemma}

\begin{proof}
The fact that $\sigma_1(TU)\ge\sigma_2(T)\sigma_1(U)$ gives the claim  
for $\alpha\le 1$. If $\alpha>1$, write $\beta=2-\alpha$. Since 
$\det(T)=\sigma_1(T)\sigma_2(T)$ and $\sigma_1(T^{-1})=\sigma_2(T)^{-1}$, 
we obtain
\begin{equation}\label{inverse}
\Phi^\alpha(T)=\sigma_1(T^{-1})^\beta\det(T)\text{ and }
\underline\Phi^\alpha(T)=\sigma_2(T^{-1})^\beta\det(T),
\end{equation}
implying
\begin{equation*}
\begin{aligned}
\Phi^\alpha(TU)=\sigma_1(U^{-1}T^{-1})^\beta\det(TU)
   &\ge\sigma_2(T^{-1})^\beta\det(T)\sigma_1(U^{-1})^\beta\det(U)\\
   &=\underline\Phi^\alpha(T)\Phi^\alpha(U).
\end{aligned}
\end{equation*}
\end{proof}

\begin{remark}\label{lowerpressure}
Observation \eqref{inverse} implies that $\underline\Phi^\alpha$ is 
supermultiplicative, that is, 
$\underline\Phi^\alpha(TU)\ge\underline\Phi^\alpha(T)\underline\Phi^\alpha(U)$
for all linear maps $T,U\colon\mathbb R^2\to\mathbb R^2$. Similarly
as in the proof of Theorem~\ref{pexists} this in turn implies that for 
$P$-almost all 
$\tilde\omega\in\widetilde\Omega$ the limit
\[
\lim_{n\to\infty}\frac 1n\log\sum_{\bi_{N_n}\in\Sigma_*^{\tilde\omega}(0,n)}
  \underline\Phi^\alpha(T_{\bi_{N_n}}^{\tilde\omega})
  =:\underline{\tilde p}^{\tilde\omega}(\alpha)\le\tilde p^{\tilde\omega}(\alpha)
\]
exists for all $\alpha\in[0,2]$.
\end{remark}

\begin{lemma}\label{poslowpressure}
Inequality \eqref{eq66} holds for $P$-almost all 
$\tilde\omega\in\widetilde\Omega$ for which 
$\underline{\tilde p}^{\tilde\omega}(\alpha)>0$.
\end{lemma}

\begin{proof}
Let $\tilde\omega\in\widetilde\Omega$ satisfy 
$\underline{\tilde p}^{\tilde\omega}(\alpha)>0$ and \eqref{smallstep}.  
There exists $\zeta>1$ such that for all sufficiently large $n\in\mathbb N$ 
\[
\sum_{\bi_{N_n}\in\Sigma_*^{\tilde\omega}(0,n)}
\underline\Phi^\alpha(T_{\bi_{N_n}}^{\tilde\omega})\ge\zeta^n.
\]
Hence, inequality \eqref{measurebound}, Lemma~\ref{philowerbound} and 
the fact $\overline\sigma<1$ yield 
\[
\mu_{n+m}^{\tilde\omega}([\bi_l])
 \le M^{N_n-N_{n-1}}\zeta^{-n}\Phi^\alpha(T_{\bi_l}^{\tilde\omega})
\]
for all sufficiently large 
$n\in\mathbb N$, for all $m\in\mathbb N$ and for all $N_{n-1}\le l<N_n$.
By \eqref{smallstep}, for every $\varepsilon>0$ there exists
$n_\varepsilon$ such that $N_n-N_{n-1}\le\varepsilon n$ 
for all $n\ge n_\varepsilon$. This implies the existence of a constant 
$c(\tilde\omega)<\infty$ such that 
\[
\mu_{n+m}^{\tilde\omega}([\bi_l])
  \le c(\tilde\omega)\Phi^\alpha(T_{\bi_l}^{\tilde\omega})
\]
for all $n,m\in\mathbb N$ and $N_{n-1}\le l<N_n$. Since $[\bi_l]$ is open, 
we have
$\mu^{\tilde\omega}([\bi_l])\le\liminf_{k\to\infty}
  \mu_{n+m_k}^{\tilde\omega}([\bi_l])$, where 
$(\mu_{n+m_k})$ is the subsequence converging to $\mu^{\tilde\omega}$. 
This completes the proof.
\end{proof}

For the remaining two cases of the proof of inequality \eqref{eq66}, we need 
the following notation. Let 
$T\colon\mathbb R^2\to\mathbb R^2$ be a linear map.
If $\sigma_1(T)>\sigma_2(T)$ there exist unique (up to sign) unit vectors 
$v_1(T)$ and $v_2(T)$ such that $|T(v_i)|=\sigma_i$ for $i=1,2$. Set 
$w_i(T)=T(v_i)$ for $i=1,2$. Observe that $v_1(T)$ and
$v_2(T)$ are perpendicular to each other since they are eigenvectors of the
self-adjoint map $T^*T$, and $w_1(T)$ and $w_2(T)$ are also perpendicular 
as the semi-axes of an ellipse. If $\sigma_1(T)=\sigma_2(T)$, let 
$v_1(T)$ and $v_2(T)$ be any pair of perpendicular unit vectors. Given two  
linear maps $T,U\colon\mathbb R^2\to\mathbb R^2$, write 
\[
\widehat w_1(U):=\frac{w_1(U)}{|w_1(U)|}=av_1(T)+bv_2(T),
\]
where $a$ and $b$ depend on $T$ and $U$. Then 
\[
\sigma_1(TU)\ge|T(w_1(U))|\ge |a|\sigma_1(T)\sigma_1(U)
\] 
and the fact
$\sigma_2(TU)\ge\sigma_2(T)\sigma_2(U)$ gives
\begin{equation}\label{phibound}
\Phi^\alpha(TU)\ge |a|^\alpha\Phi^\alpha(T)\Phi^\alpha(U).
\end{equation}

Since $\alpha<\alpha_0$, we have $\tilde p^{\tilde\omega}(\alpha)>0$, and by 
Lemma~\ref{poslowpressure}, it remains to consider the case 
$\underline{\tilde p}^{\tilde\omega}(\alpha)<\tilde p^{\tilde\omega}(\alpha)$. 
For every such $\tilde\omega\in\widetilde\Omega$, there exist $\lambda>1$, 
$0<\rho<\lambda$, $0<\gamma<1$ and $n_0\in\mathbb N$ such that 
\begin{equation}\label{n0bound}
\sum_{\bi_{N_n}\in\Sigma_*^{\tilde\omega}(0,n)}\Phi^\alpha
  (T_{\bi_{N_n}}^{\tilde\omega})>\lambda^n\text{ and }
\sum_{\bi_{N_n}\in\Sigma_*^{\tilde\omega}(0,n)}\underline\Phi^\alpha
  (T_{\bi_{N_n}}^{\tilde\omega})<(\gamma\rho)^n
\end{equation}
for all $n\ge n_0$. Since the set
\[
\widetilde\Omega^<=\{\tilde\omega\in\widetilde\Omega\mid
  \underline{\tilde p}^{\tilde\omega}(\alpha)<\tilde p^{\tilde\omega}(\alpha)\}
\]
can be represented as a countable union of sets consisting of points which
satisfy \eqref{n0bound} with rational 
$\lambda$, $\rho$ and $\gamma$, we may fix $\lambda>1$, 
$0<\rho<\lambda$ and $0<\gamma<1$ for the rest of the proof. We denote the 
corresponding set by $\widetilde\Omega_{\lambda,\rho,\gamma}^<$,
and  let
\[
E_{n_0}=\{\tilde\omega\in\widetilde\Omega_{\lambda,\rho,\gamma}^<\mid
  \eqref{n0bound}\text{ is valid for all }n\ge n_0\}.
\] 
Since 
$\bigcup_{n_0=1}^\infty E_{n_0}=\widetilde\Omega_{\lambda,\rho,\gamma}^<$, 
we may also fix $n_0\in\mathbb N$ for the rest of the proof.

Let $1<\lambda_1<\lambda$ and $\frac{\lambda_1}\lambda<\gamma_1<1$. 
For all
$n,m\in\mathbb N$, $\bi_{N_n}\in\Sigma_*^{\tilde\omega}(0,n)$ and 
$\bk_{N_m}\in\Sigma_*^{\tilde\omega}(n,n+m)$, let 
\begin{equation}\label{w1decomb}
\widehat w_1(T_{\bk_{N_m}}^{\Xi^n(\tilde\omega)})
  =av_1(T_{\bi_{N_n}}^{\tilde\omega})+bv_2(T_{\bi_{N_n}}^{\tilde\omega}),
\end{equation}
where $a$ and $b$ are functions of $\bi_{N_n}$ and $\bk_{N_m}$, and define
\[
C_n^{\tilde\omega}(\bk_{N_m})=\{\bi_{N_n}\in\Sigma_*^{\tilde\omega}(0,n)\mid 
  |a|>\lambda_1^{-\frac n\alpha}\}.
\]

\begin{lemma}\label{notaligned}
Let $\underline{\tilde p}^{\tilde\omega}(\alpha)<\tilde p^{\tilde\omega}(\alpha)$.
Assuming that for all large $n,m\in\mathbb N$ we have
\begin{equation}\label{bigportion}
\sum_{\bi_{N_n}\in C_n^{\tilde\omega}(\bk_{N_m})}\Phi^\alpha
  (T_{\bi_{N_n}}^{\tilde\omega})>\gamma_1^n
  \sum_{\bi_{N_n}\in\Sigma_*^{\tilde\omega}(0,n)}
  \Phi^\alpha(T_{\bi_{N_n}}^{\tilde\omega})
\end{equation}
for all $\bk_{N_m}\in\Sigma_*^{\tilde\omega}(n,n+m)$, inequality 
\eqref{eq66} is valid.
\end{lemma}

\begin{proof}
Let $n\ge n_0$ and $m\in\mathbb N$ be so large that inequality 
\eqref{bigportion} holds. 
Consider $N_{n-1}\le l<N_n$.
From \eqref{measurebound},
 \eqref{phibound}, assumption \eqref{bigportion} and
the first inequality in \eqref{n0bound}
we get
\[
\mu_{n+m}^{\tilde\omega}([\bi_l])\le\bigl(\frac{\lambda_1}\lambda\bigr)^n
  M^{N_n-N_{n-1}}\gamma_1^{-n}
  \Phi^\alpha(T_{\bi_l}^{\tilde\omega}).
\]
Proceeding  as in the proof of
Lemma~\ref{poslowpressure}, it follows that \eqref{eq66} holds.
  
\end{proof}

Below we will prove that the assumption in \eqref{bigportion} holds for 
$P$-almost all ${\tilde\omega} \in \widetilde\Omega^<$. 
The heuristic idea of the proof is as follows: If the assumption in 
\eqref{bigportion} is false,
then the longer semiaxes of most of the ellipses
$T_{\bk_{N_m}}^{\Xi^n(\tilde\omega)}(B(0,1))$ are almost 
parallel to $v_2(T_{\bi_{N_n}}^{\tilde\omega})$ for most $\bi_{N_n}$. In 
particular, most of $v_2(T_{\bi_{N_n}}^{\tilde\omega})$ are nearly parallel to
each other. By assumption \eqref{decay}, the distance $N_n-N_{n-1}$ is small 
compared to $n$, and therefore, $v_2(T_{\bi_{N_{n-1}}}^{\tilde\omega})$ is roughly
parallel to the image of $v_2(T_{\bi_{N_n}}^{\tilde\omega})$ under
$T_{\bi_{N_1}}^{\Xi^{n-1}(\tilde\omega)}$. In this manner one 
finds a vector which is mapped in the same way by all the maps between the
levels $N_{n-1}$ and $N_n$, and by assumption \eqref{direction}, 
the probability of this event is smaller than $\eta<1$. 
Selecting an integer $h$ in a suitable manner and repeating
this argument in the blocks $[N_{n-h},N_{n-h+1}],\dots,[N_{n-1},N_n]$
one may attach a vector $v_t$ to each $[N_{n-t},N_{n-t+1}]$ such that $v_t$
is mapped in the same manner by all the maps between the
levels $N_{n-t}$ and $N_{n-t+1}$.
By independence, this happens with probability 
less than $\eta^h$, and the Borel-Cantelli lemma implies that this may
 happen only
for finitely many $n\in\mathbb N$.  

We will now make this heuristic idea precise.
Consider $c_0>0$. For $L\in\mathbb N$, denote by 
$\widetilde\Omega(L)$ the set of $\tilde\omega\in\widetilde\Omega$ for which
\begin{equation}\label{control}
N_{l+1}-N_l\le\frac{c_0l}{\log l} 
\end{equation}
for all $l\geq L$. Since $N_1\circ\Xi^n=N_{n+1}-N_n$ and $P$ is 
$\Xi$-invariant, 
assumption \eqref{decay} combined with Borel-Cantelli lemma implies that 
for $P$-almost all $\tilde\omega\in\widetilde\Omega$ there is an 
$L\in\mathbb N$ such that $\tilde\omega\in\widetilde\Omega(L)$. 
Fix $L\in\mathbb N$ for the rest of the proof.

By considering a decreasing sequence of events, it follows
from assumption \eqref{direction} that there exist $\delta>0$ and $\eta<1$ 
such that
\begin{equation}\label{parallel}
\begin{aligned}
P\bigl(\{\tilde\omega\in\widetilde\Omega\mid&\text{ there exists }
  v\in\mathbb R^2\setminus\{0\}\text{ such that }
\sphericalangle(T_{\bi_{N_1}}^{\tilde\omega}(v),T_{\bi'_{N_1}}^{\tilde\omega}(v))
 \le\delta\\
&\text{ for all }\bi_{N_1},\bi'_{N_1}\in\Sigma_*^{\tilde\omega}(0,1)\}
 \bigr)<\eta,
\end{aligned}
\end{equation}
where $\sphericalangle(v,u)$ is the angle between 
$u,v\in\mathbb R^2\setminus\{0\}$. For $n\in\mathbb N$, let 
\begin{equation*}
\begin{aligned}
H_n=\{\tilde\omega\in\widetilde\Omega\mid&\text{ there exists }
  v\in\mathbb R^2\setminus\{0\}\text{ such that }
\sphericalangle(T_{\bi_{N_1}}^{\Xi^{n-1}(\tilde\omega)}(v),
 T_{\bi'_{N_1}}^{\Xi^{n-1}(\tilde\omega)}(v))
 \le\delta\\
&\text{ for all }\bi_{N_1},\bi'_{N_1}\in\Sigma_*^{\tilde\omega}(n-1,n)\}.
\end{aligned}
\end{equation*}
Since $H_n=\Xi^{-(n-1)}(H_1)$, we have $P(H_n)<\eta$ for all $n\in\mathbb N$.
Furthermore, by assumption, the events $H_n$ and $H_m$ are independent 
for $n\ne m$. Observe
that $H_n$ is independent of $\gamma$, $\gamma_1$, $\lambda$, $\lambda_1$,
$\rho$, $n_0$ and $L$.

For all $n,m\in\mathbb N$ with $n\ge n_0$, set
\begin{equation*}
\begin{aligned}
G_{n,m}=\{\tilde\omega\in E_{n_0}\mid&\text{ there exists }\bk_{N_m}\in
  \Sigma_*^{\tilde\omega}(n,n+m)\text{ for which}\\
  &\text{ inequality \eqref{bigportion} is not valid}\},
\end{aligned}
\end{equation*}
and define $G_n=\bigcap_{k=1}^\infty\bigcup_{m=k}^\infty G_{n,m}$. Let
\[
B_n^{\tilde\omega}=\{\bi_{N_n}\in\Sigma_*^{\tilde\omega}(0,n)\mid
\frac{\sigma_2(T_{\bi_{N_n}}^{\tilde\omega})}{\sigma_1(T_{\bi_{N_n}}^{\tilde\omega})}
\ge\bigl(\frac\rho\lambda\bigr)^{\frac n\beta}\},
\]
where $\beta=\alpha$ for $\alpha\le 1$ and $\beta=2-\alpha$ for $\alpha>1$. 
From \eqref{n0bound} one obtains
\begin{equation}\label{exceptional}
\sum_{\bi_{N_n}\in B_n^{\tilde\omega}}\Phi^\alpha(T_{\bi_{N_n}}^{\tilde\omega})
<\gamma^n\sum_{\bi_{N_n}\in\Sigma_*^{\tilde\omega}(0,n)}
\Phi^\alpha(T_{\bi_{N_n}}^{\tilde\omega})
\end{equation}
for all $n\ge n_0$.
 
By Lemmas \ref{poslowpressure} and \ref{notaligned} it is enough to prove that
for $P$-almost all $\tilde\omega\in E_{n_0}$ the number of those $n\ge n_0$, for
which $\tilde\omega\in G_n$, is finite, that is, 
$P(\bigcap_{k=n_0}^\infty\bigcup_{n=k}^\infty G_n)=0$. Let $n\ge n_0$ and 
$\tilde\omega\in G_n$. Fix $m\in\mathbb N$ and 
$\bk_{N_m}\in\Sigma_*^{\tilde\omega}(n,n+m)$ for which inequality
\eqref{bigportion} is not valid. Choose $h\in\mathbb N$ with $h\le h_n$ where
$h_n=\lfloor C\log n\rfloor$ is the integer part of $C\log n$ for some large 
$C$ to be fixed later.
We claim that for large enough $n\in\mathbb N$ there is
$\bi_{N_{n-h}}\in\Sigma_*^{\tilde\omega}(0,n-h)\setminus B_{n-h}^{\tilde\omega}$ 
such that 
$\bi_{N_{n-h}}\bj_{N_h}\in\Sigma_*^{\tilde\omega}(0,n)\setminus 
  C_n^{\tilde\omega}(\bk_{N_m})$ 
for all $\bj_{N_h}\in\Sigma_*^{\tilde\omega}(n-h,n)$. Indeed, if for every 
$\bi_{N_{n-h}}\in\Sigma_*^{\tilde\omega}(0,n-h)\setminus B_{n-h}^{\tilde\omega}$ 
there exists $\bj_{N_h}\in\Sigma_*^{\tilde\omega}(n-h,n)$ such that 
$\bi_{N_{n-h}}\bj_{N_h}\in C_n^{\tilde\omega}(\bk_{N_m})$, inequalities
\eqref{exceptional} and \eqref{control} yield
\begin{equation}\label{allsame}
\begin{aligned}
\sum_{\bi_{N_n}\in C_n^{\tilde\omega}(\bk_{N_m})}\Phi^\alpha
 (&T_{\bi_{N_n}}^{\tilde\omega})
\ge\sum_{\bi_{N_{n-h}}\in\Sigma_*^{\tilde\omega}(0,n-h)\setminus 
 B_{n-h}^{\tilde\omega}}
 \Phi^\alpha(T_{\bi_{N_{n-h}}}^{\tilde\omega}T_{\bj_{N_h}}^{\Xi^{n-h}(\tilde\omega)})\\
&\ge(\underline\sigma^\alpha)^{N_n-N_{n-h}}(1-\gamma^{n-h})
 \sum_{\bi_{N_{n-h}}\in\Sigma_*^{\tilde\omega}(0,n-h)}
 \Phi^\alpha(T_{\bi_{N_{n-h}}}^{\tilde\omega})\\
&\ge\bigl(\frac{\underline\sigma^\alpha}{\overline\sigma^\alpha M}
  \bigr)^{hc_0\frac n{\log n}}(1-\gamma^{n-h})
  \sum_{\bi_{N_n}\in\Sigma_*^{\tilde\omega}(0,n)}\Phi^\alpha
  (T_{\bi_{N_n}}^{\tilde\omega}).
\end{aligned}
\end{equation}
Choosing sufficiently small $c_0>0$ and sufficiently large $n\in\mathbb N$
(recall $h\le C\log n$), this contradicts the assumption 
that inequality \eqref{bigportion} is not valid. 

Consider  
$\bi_{N_{n-h}}\in\Sigma_*^{\tilde\omega}(0,n-h)\setminus B_{n-h}^{\tilde\omega}$
and $\bj_{N_h}\in\Sigma_*^{\tilde\omega}(n-h,n)$ such that
$\bi_{N_{n-h}}\bj_{N_h}\in\Sigma_*^{\tilde\omega}(0,n)\setminus 
  C_n^{\tilde\omega}(\bk_{N_m})$.
For any unit vector $v\in S^1\subset\mathbb R^2$, write
\begin{equation}\label{v2decomb}
\frac{T_{\bj_{N_h}}^{\Xi^{n-h}(\tilde\omega)}(v)}
   {|T_{\bj_{N_h}}^{\Xi^{n-h}(\tilde\omega)}(v)|}
  =c(v)v_1(T_{\bi_{N_{n-h}}}^{\tilde\omega})
   +d(v)v_2(T_{\bi_{N_{n-h}}}^{\tilde\omega}).
\end{equation}
Defining 
$v_2^{-1}=\frac{(T_{\bj_{N_h}}^{\Xi^{n-h}(\tilde\omega)})^{-1}
  (v_2(T_{\bi_{N_{n-h}}}^{\tilde\omega}))}
  {|(T_{\bj_{N_h}}^{\Xi^{n-h}(\tilde\omega)})^{-1}
  (v_2(T_{\bi_{N_{n-h}}}^{\tilde\omega}))|}\in S^1$, 
we observe that
\begin{equation*}
\begin{aligned}
|c(v_2&(T_{\bi_{N_{n-h}}\bj_{N_h}}^{\tilde\omega}))|\,\,\sigma_1
  (T_{\bi_{N_{n-h}}}^{\tilde\omega})\,\,|T_{\bj_{N_h}}^{\Xi^{n-h}(\tilde\omega)}
  (v_2(T_{\bi_{N_{n-h}}\bj_{N_h}}^{\tilde\omega}))|\\
&\le |T_{\bi_{N_{n-h}}\bj_{N_h}}^{\tilde\omega}
  (v_2(T_{\bi_{N_{n-h}}\bj_{N_h}}^{\tilde\omega}))|
  \le |T_{\bi_{N_{n-h}}\bj_{N_h}}^{\tilde\omega}(v_2^{-1})|
=\sigma_2(T_{\bi_{N_{n-h}}}^{\tilde\omega})
  |T_{\bj_{N_h}}^{\Xi^{n-h}(\tilde\omega)}(v_2^{-1})|.
\end{aligned}
\end{equation*}
Since $\bi_{N_{n-h}}\not\in B_{n-h}^{\tilde\omega}$, we have for all 
$n\ge L-h$
\begin{equation}\label{anglebound}
\begin{aligned}
|c(v_2(T_{\bi_{N_{n-h}}\bj_{N_h}}^{\tilde\omega}))|
 &\le\frac{\sigma_2(T_{\bi_{N_{n-h}}}^{\tilde\omega})} {\sigma_1
   (T_{\bi_{N_{n-h}}}^{\tilde\omega})}\frac{|T_{\bj_{N_h}}^{\Xi^{n-h}(\tilde\omega)}
   (v_2^{-1})|}{|T_{\bj_{N_h}}^{\Xi^{n-h}(\tilde\omega)}
   (v_2(T_{\bi_{N_{n-h}}\bj_{N_h}}^{\tilde\omega}))|}\\
 &<\bigl(\frac\rho\lambda\bigr)^{\frac{n-h}\beta}\bigl(\frac{\overline\sigma}
   {\underline\sigma}\bigr)^{hc_0\frac n{\log n}}.
\end{aligned}
\end{equation}
We continue by showing that 
$T_{\bj_{N_h}}^{\Xi^{n-h}(\tilde\omega)}
 (\widehat w_1(T_{\bk_{N_m}}^{\Xi^n(\tilde\omega)}))$ and 
$T_{\bj'_{N_h}}^{\Xi^{n-h}(\tilde\omega)}
 (\widehat w_1(T_{\bk_{N_m}}^{\Xi^n(\tilde\omega)}))$
are nearly parallel to each other for all
$\bj_{N_h},\bj'_{N_h}\in\Sigma_*^{\tilde\omega}(n-h,n)$ 
when $n\in\mathbb N$ is sufficiently large.
Indeed, recalling \eqref{w1decomb} and \eqref{v2decomb}, we have
\begin{equation*}
\begin{aligned}
&T_{\bj_{N_h}}^{\Xi^{n-h}(\tilde\omega)}
(\widehat w_1(T_{\bk_{N_m}}^{\Xi^n(\tilde\omega)}))
   =a\,T_{\bj_{N_h}}^{\Xi^{n-h}(\tilde\omega)}(
   v_1(T_{\bi_{N_{n-h}}\bj_{N_h}}^{\tilde\omega}))\\
 &+b\,|T_{\bj_{N_h}}^{\Xi^{n-h}(\tilde\omega)}(
   v_2(T_{\bi_{N_{n-h}}\bj_{N_h}}^{\tilde\omega}))|
   \bigl(c(v_2(T_{\bi_{N_{n-h}}\bj_{N_h}}^{\tilde\omega}))
   (v_1(T_{\bi_{N_{n-h}}}^{\tilde\omega})
   +d(v_2(T_{\bi_{N_{n-h}}\bj_{N_h}}^{\tilde\omega}))
   v_2(T_{\bi_{N_{n-h}}}^{\tilde\omega})\bigr),
\end{aligned}
\end{equation*}
and thus
\begin{equation*}
\begin{aligned}
|\tan(\sphericalangle\bigl(T_{\bj_{N_h}}^{\Xi^{n-h}(\tilde\omega)}
   (\widehat w_1(T_{\bk_{N_m}}^{\Xi^n(\tilde\omega)})),
   &v_2(T_{\bi_{N_{n-h}}}^{\tilde\omega})\bigr))|\\ 
&\leq\Big|\frac{c(v_2(T_{\bi_{N_{n-h}}\bj_{N_h}}^{\tilde\omega}))}
   {d(v_2(T_{\bi_{N_{n-h}}\bj_{N_h}}^{\tilde\omega}))}\Big| +
   \Big|\frac{ea}{bd(v_2(T_{\bi_{N_{n-h}}\bj_{N_h}}^{\tilde\omega}))}\Big|
   (\frac{\overline\sigma}{\underline\sigma}\bigr)^{{N_h}}, 
\end{aligned}
\end{equation*}
where $b$ and $d(v_2(T_{\bi_{N_{n-h}}\bj_{N_h}}^{\tilde\omega}))$ and $e$ are 
close to one. The factor $e$ appears since 
$T_{\bj_{N_h}}^{\Xi^{n-h}(\tilde\omega)}
  (v_1(T_{\bi_{N_{n-h}}\bj_{N_h}}^{\tilde\omega}))$
need not to be perpendicular to $v_2(T_{\bi_{N_{n-h}}}^{\tilde\omega})$.
By using the definition of $C_n^{\tilde\omega}(\bk_{N_m})$
and inequality \eqref{anglebound} we deduce that for sufficiently large 
$n\in\mathbb N$ 
\begin{equation}\label{smallanglegen}
\begin{aligned}
\sphericalangle\bigl(&T_{\bj_{N_h}}^{\Xi^{n-h}(\tilde\omega)}
   (\widehat w_1(T_{\bk_{N_m}}^{\Xi^n(\tilde\omega)})),
   T_{\bj'_{N_h}}^{\Xi^{n-h}(\tilde\omega)}(
   \widehat w_1(T_{\bk_{N_m}}^{\Xi^n(\tilde\omega)}))\bigr)\\
 &\le 2\max_{\bj=\bj_{N_h},\hspace{1mm} \bj'_{N_h}}\sphericalangle\bigl(
  T_\bj^{\Xi^{n-h}(\tilde\omega)}(\widehat w_1(T_{\bk_{N_m}}^{\Xi^n(\tilde\omega)})),
   v_2(T_{\bi_{N_{n-h}}}^{\tilde\omega})\bigr)
 \le 5\gamma_2^{n-h}\bigl(\frac{\overline\sigma}
   {\underline\sigma}\bigr)^{hc_0\frac n{\log n}}
\end{aligned}
\end{equation}
for all $\bj_{N_h},\bj'_{N_h}\in\Sigma_*^{\tilde\omega}(n-h,n)$, where
$\gamma_2=\max\{\lambda_1^{-\frac 1\alpha},
  \bigl(\frac\rho\lambda\bigr)^{\frac 1\beta}\}$.

Writing $\bj_{N_h}=\bj_{N_1}\bj_{N_{h-1}}$, we get
\[
T_{\bj_{N_h}}^{\Xi^{n-h}(\tilde\omega)}(w_1(T_{\bk_{N_m}}^{\Xi^n(\tilde\omega)}))
=T_{\bj_{N_1}}^{\Xi^{n-h}(\tilde\omega)}(T_{\bj_{N_{h-1}}}^{\Xi^{n-h+1}(\tilde\omega)}
(w_1(T_{\bk_{N_m}}^{\Xi^n(\tilde\omega)}))).
\]
Now for any fixed $\bj_{N_{h-1}}\in\Sigma_*^{\tilde\omega}(n-h+1,n)$ inequality
\eqref{smallanglegen} implies
\begin{equation}\label{smallangle}
\begin{aligned}
\sphericalangle(&T_{\bj_{N_1}}^{\Xi^{n-h}(\tilde\omega)}
  (T_{\bj_{N_{h-1}}}^{\Xi^{n-h+1}(\tilde\omega)}
  (w_1(T_{\bk_{N_m}}^{\Xi^n(\tilde\omega)}))),
  T_{\bj'_{N_1}}^{\Xi^{n-h}(\tilde\omega)}(T_{\bj_{N_{h-1}}}^{\Xi^{n-h+1}
  (\tilde\omega)}(w_1(T_{\bk_{N_m}}^{\Xi^n(\tilde\omega)}))))\\
&\le 5\gamma_2^{n-h}\bigl(\frac{\overline\sigma}
   {\underline\sigma}\bigr)^{hc_0\frac n{\log n}}
\end{aligned}
\end{equation} 
for all $\bj_{N_1},\bj'_{N_1}\in\Sigma_*^{\tilde\omega}(n-h,n-h+1)$. 
Let $n_\delta\in\mathbb N$ be such that 
$5\gamma_2^{n_\delta-h_{n_\delta}}\bigl(\frac{\overline\sigma}
 {\underline\sigma}\bigr)^{h_{n_\delta}c_0\frac{n_\delta}{\log n_\delta}}<\delta$.
From \eqref{smallangle} we see that for large enough $n\ge n_\delta$ we have 
$\tilde\omega\in H_{n-h+1}$ for all $h\le h_n$. By independence, inequality 
\eqref{parallel} implies that for all such $n\ge n_\delta$
\[
P(G_n)\le P(\bigcap_{l-h_n+1}^n H_l)<\eta^{h_n}.
\] 

Select sufficiently large $C$ and sufficiently small $c_0$ such that
\[
\eta^C<e^{-1},\quad
\sqrt{\gamma_2}(\frac{\overline\sigma}{\underline\sigma})^{Cc_0}<1\quad
\text{and}\quad
\gamma_1<(\frac{\underline\sigma^\alpha}{\overline\sigma^\alpha M})^{Cc_0}.
\] 
Moreover, let $n\in\mathbb N$ be so large that 
\[
n-C\log n>\frac n2,\quad n\ge\max\{2L,n_\delta,2n_0\}\quad\text{and}\quad 
\bigl(\gamma_1(\frac{\overline\sigma^\alpha M}{\underline\sigma^\alpha})^{Cc_0}
  \bigr)^n<1-\sqrt\gamma.
\] 
Observing that for such $n$ all the arguments above (see 
\eqref{allsame}, \eqref{anglebound} and \eqref{smallangle}) hold, gives 
$P(G_n)\le\eta^{\lfloor C\log n\rfloor}$. Since 
$\sum_{n=1}^\infty\eta^{C\log n}<\infty$, 
Borel-Cantelli lemma implies that for $P$-almost all 
$\tilde\omega\in\widetilde\Omega(L)\cap E_{n_0}$ we have $\tilde\omega\in G_n$ 
only for finitely many $n\in\mathbb N$. Since this is true for all 
$L,n_0\in\mathbb N$ (and for all rational $1<\lambda_1<\lambda$, 
$0<\rho<\lambda$, $0<\gamma<1$ and $\frac{\lambda_1}\lambda<\gamma_1<1$),
inequality \eqref{eq66} is valid for 
$P$-almost all $\tilde\omega\in\widetilde\Omega$.
This completes the proof of Theorem~\ref{main}.
\end{proof} 

We conclude this section by discussing the assumptions of Theorem~\ref{main}.

\begin{remark}\label{assumptions}
We say that a family $F^\lambda=\{f_1,\dots,f_n\}$ is parallel if
there exists $v$ such that $f_1(v),\dots,f_n(v)$ are parallel. If a family 
$F^\lambda$ is not parallel, it is irreducible in the sense of \cite{BL}.
Note that the complement of the set of parallel families is open in any 
reasonable metric. The set of non-parallel families is also
dense in the following sense: Suppose $F^\lambda=\{f_1,\dots,f_n\}$ is parallel.
Then the family
$F^{\lambda'}=\{f_1,\dots,f_n,R_\varepsilon\circ f_1\}$, where $R_\varepsilon$ 
is a rotation by angle $\varepsilon$, is not parallel. If we view 
$F^\lambda$ as a degenerate family $\{f_1,\dots,f_n,f_1\}$, then 
$F^{\lambda'}$ is close to $\{f_1,\dots,f_n,f_1\}$. 
Assumption \eqref{direction} is thus weak in the sense that if it is not 
satisfied then after a small perturbation in the family {\bf F} and in the
measure $P$ it is satisfied. In particular, if all the families in
{\bf F} are non-parallel, then \eqref{direction} is valid for any $P$ for 
which $P(\{\tilde\omega\in\widetilde\Omega\mid N_1(\tilde\omega)=1\})>0$. 
Further, the 
validity of \eqref{direction} is not destroyed by small perturbations. 
Of course, it should be noted that ergodicity is not necessarily preserved 
under perturbations.

Assumption \eqref{direction} is not necessary for the validity of 
Theorem~\ref{main}. Indeed, it is not difficult 
to show that Theorem~\ref{main} still holds if assumption \eqref{direction} is 
replaced by the assumption that all the maps in 
{\bf F} fix two given directions. Therefore, the deterministic iterated 
function system considered in \cite[Example 2]{E} (see also \cite{Er,PU,SS}) 
is a special case of our theorem. Hence, in Theorem~\ref{main}, the assumption 
$\overline\sigma<\frac 12$ is necessary. 

By  \eqref{smallstep}, the condition 
$\int_{\widetilde\Omega}N_1(\tilde\omega)\,dP(\tilde\omega)<\infty$ implies that
for any $\varepsilon>0$ we have $N_{l+1}-N_l<\varepsilon l$ for large enough 
$l\in\mathbb N$. Thus by \eqref{control} assumption \eqref{decay} is only 
slightly stronger than the condition 
$\int_{\widetilde\Omega}N_1(\tilde\omega)\,dP(\tilde\omega)<\infty$.

The strongest assumptions in Theorem~\ref{main} are the independence
between neck levels and the condition $D=2$. In the proof the role of 
independence is quite clear whilst the restriction $D=2$ is more 
hidden. The latter assumption is used in the definition of the lower pressure 
$\underline{\tilde p}^{\tilde\omega}$. Clearly, the definition could be 
modified in higher dimensional case.
However, for $D\ge 3$ the main problem is that ellipsoids have more than two
semiaxes and, for example, the counterpart of inequality \eqref{anglebound} 
is not obvious.

Examples of measures satisfying assumption \eqref{decay} and the independence 
between neck levels are discussed in 
Example~\ref{assumpsat}. For example assuming exponential decay for 
$P(\{\tilde\omega\in\widetilde\Omega\mid N_1(\tilde\omega)\ge t\})$ 
implies \eqref{decay}.
\end{remark}

According to the following proposition for self-similar 
code tree fractals the assumptions of Theorem~\ref{pexists} (with 
$\overline\sigma<1$ replaced by $\overline\sigma<\frac 12$)
are sufficient for the validity of the dimension formula.

\begin{proposition}\label{selfsimilar}
Let $P$ be an ergodic $\Xi$-invariant probability measure on 
$\widetilde\Omega$ such that $\int_{\widetilde\Omega}N_1(\tilde\omega)\,
dP(\tilde\omega)<\infty$. Assume that ${\bf F}$ is a family of iterated
function systems consisting of similarities
on $\mathbb R^D$ and $0<\underline\sigma\le\overline\sigma<\frac 12$. Then for 
$P$-almost all $\tilde\omega\in\widetilde\Omega$ 
\[
\dimH(A_{\ba}^{\tilde\omega})=\min\{\alpha_0,D\}
\]
for $\mathcal L^{D\mathcal A}$-almost all $\ba\in\mathbb R^{D\mathcal A}$.
\end{proposition}

\begin{proof}
Since the mappings are similarities, $\Phi^\alpha$ is multiplicative. We may  
proceed as in the proof of Lemma~\ref{poslowpressure} to verify that 
the measure constructed in \eqref{muomega} satisfies  \eqref{eq66}.
The conclusion of  Proposition \ref{selfsimilar} then follows as in the 
proof of Theorem~\ref{0510} as described in the beginning of the the proof of 
Theorem~\ref{main}.
\end{proof}

\begin{remark}\label{UOSCresult}
Combining \eqref{eq66} with the methods of \cite{BHS08}, we obtain a 
generalization of the main theorem in \cite{BHS08}, that is, under the 
assumptions of Theorem~\ref{pexists}
and assuming that the maps are similarities, the dimension formula
is valid $P$-almost surely for any $\ba\in\mathbb R^{D\mathcal A}$ for which the
uniform open set condition \eqref{UOSC2} is satisfied. 
\end{remark}

\end{document}